\documentclass[10pt,a4paper,reqno]{amsart}

\usepackage[textwidth=145mm,textheight=220mm,centering]{geometry}
\usepackage{amsmath,amssymb,amsthm,mathtools}
\usepackage{enumitem}
\usepackage{microtype}
\usepackage[hidelinks]{hyperref}

\newtheoremstyle{nmjplain}{6pt}{6pt}{\itshape}{}{\bfseries}{.}{.5em}{}
\newtheoremstyle{nmjdefinition}{6pt}{6pt}{}{}{\bfseries}{.}{.5em}{}
\newtheoremstyle{nmjremark}{6pt}{6pt}{}{}{\bfseries}{.}{.5em}{}
\theoremstyle{nmjplain}
\newtheorem{theorem}{Theorem}[section]
\newtheorem{maintheorem}{Theorem}

\newtheorem{proposition}[theorem]{Proposition}
\newtheorem{lemma}[theorem]{Lemma}
\newtheorem{corollary}[theorem]{Corollary}
\theoremstyle{nmjdefinition}
\newtheorem{definition}[theorem]{Definition}
\newtheorem{notation}[theorem]{Notation}
\newtheorem{example}[theorem]{Example}
\newtheorem{problem}[theorem]{Problem}
\theoremstyle{nmjremark}
\newtheorem{remark}[theorem]{Remark}

\newcommand{\supp}{\operatorname{supp}}
\newcommand{\Rep}{\operatorname{Rep}}
\newcommand{\End}{\operatorname{End}}
\newcommand{\pd}{\operatorname{pd}}
\newcommand{\depth}{\operatorname{depth}}

\newcommand{\im}{\operatorname{im}}
\newcommand{\coker}{\operatorname{coker}}
\newcommand{\Dgr}{D_{\mathrm{gr}}}

\title[Homological Structure of Local Cohomology]{Multigraded Local Cohomology Categories:\\
	Homological Structure and Functorial Sections}
\author[Dancheng Lu]{Dancheng Lu}
\address{School of Mathematical Sciences, Soochow University,
	Suzhou 215006, P.~R.~China}
\email{ludancheng@suda.edu.cn}

\newcommand{\proj}{\operatorname{proj}}

\newcommand{\Hom}{\operatorname{Hom}}
\newcommand{\Ext}{\operatorname{Ext}}

\newcommand{\kk}{\Bbbk}

\newcommand{\LFgr}{\mathsf{Gr}_{\mathbb Z^n}^{\mathrm{lf}}(S)}

\newcommand{\modcat}{\operatorname{mod}}
\newcommand{\grade}{\operatorname{grade}}
\newcommand{\EssIm}{\operatorname{EssIm}}

\makeatletter
\renewcommand{\l@section}{\@tocline{1}{0.35em}{1pc}{}{}}
\renewcommand{\@defaultbiblabelstyle}[1]{[#1]}
\makeatother
\begin{document}
	\raggedbottom
	
	\begin{abstract}
		Let $S=\kk[x_1,\ldots,x_n]$ with its standard $\mathbb N^n$-grading,
		and let $\mathfrak m=(x_1,\ldots,x_n)$.  We study the essential images
		of the functors $M\mapsto H^i_{\mathfrak m}(M(-\mathbf1))$ on categories
		of positively determined $S$-modules, where $\mathbf1=(1,\ldots,1)$.
		For a fixed determining bound $\mathbf t\in\mathbb N^n$ and $0\le i<n$,
		we identify the image in degree $i$ with the category of finite-dimensional
		representations of the subposet of $[\mathbf0,\mathbf t]$ consisting of
		multidegrees with at least $n-i$ positive coordinates.  These images are
		Serre subcategories of the category of negatively $\mathbf t$-determined
		modules.  The top-degree image consists precisely of second cosyzygies
		in that category and is extension-closed, but is not abelian when
		$\mathbf t\ne\mathbf0$.  For the non-top image categories, we determine
		attached primes from projective covers and compute the Bass numbers at
		$\mathfrak m$ and the injective dimensions over $S$ of projective and
		injective objects.  In every non-top degree, we construct an exact
		functorial section selecting, for each nonzero target, a source of minimal
		projective dimension $n-i$ and maximal depth $i$ among all its sources.
		If $n\ge2$ and $\mathbf t\ge\mathbf1$, no fully faithful section exists
		in degrees $1$ and $n-1$, even without requiring exactness.  In top degree,
		graded Matlis duality followed by $S$-duality gives a fully faithful
		functorial section.  We also describe the global images obtained by
		varying the determining bound.  The non-top global categories have enough
		injectives and no nonzero projective objects, while the top category
		corresponds to second cosyzygies in the category of finite-support
		representations of $\mathbb N^n$.
	\end{abstract}
	
	\maketitle
	
	\begingroup
	\small
	\tableofcontents
	\endgroup
	
	\enlargethispage{2pt}
	
	\section{Introduction}\label{sec:introduction}
	
	Local cohomology modules of finitely generated multigraded modules  with respect to the maximal monomial ideal over a polynomial ring are typically not finitely generated.
	Nevertheless, the theory of finitely determined modules due to Miller~\cite{Miller} and work of Brun--Fløystad~\cite{BF} provide a way to analyze these modules using finite-dimensional representations of finite posets. This finite representation perspective raises two natural questions: which objects arise as such local cohomology modules, and how their sources can be selected.
	
   In this
	paper, we study the essential images of the (normalized) multigraded local-cohomology
	functors as full subcategories of the category of multigraded modules.
	We investigate their structure through the morphisms and extensions
	among their objects, and study whether source modules for these objects
	can be chosen functorially.  These essential image categories have not previously
	been studied in their own right.
	
	Let $S=\kk[x_1,\ldots,x_n]$ be the polynomial ring over a field $\kk$,
	with its standard $\mathbb N^n$-grading, and let
	$\mathfrak m=(x_1,\ldots,x_n)$.  Throughout this paper,
	$\mathbb N=\{0,1,2,\ldots\}$, $[n]=\{1,\ldots,n\}$,
	$\mathbf0=(0,\ldots,0)$, $\mathbf1=(1,\ldots,1)$, and
	$\mathbf e_j$ denotes the $j$th standard basis vector.  Inequalities
	between multidegrees are coordinatewise.  All morphisms of graded modules
	have degree zero, and we use the shift convention
	$M(-\mathbf c)_{\mathbf u}=M_{\mathbf u-\mathbf c}$.
	For an ideal $I\subseteq S$, let
	\[
	\Gamma_I(M)=\{m\in M:I^r m=0\text{ for some }r\ge1\}.
	\]
	We denote the $i$th right derived functor of $\Gamma_I$ by $H_I^i$.
	
	Fix $\mathbf t\in\mathbb N^n$.  A finitely generated $\mathbb Z^n$-graded $S$-module
	$M$ is called \emph{positively $\mathbf t$-determined} if it is supported
	in $\mathbb N^n$ and, for every $j\in[n]$, the multiplication map
	$x_j:M_{\mathbf a}\to M_{\mathbf a+\mathbf e_j}$ is an isomorphism
	whenever $\mathbf a\in\mathbb N^n$ and $a_j\ge t_j$.
	An Artinian $\mathbb Z^n$-graded $S$-module $N$ is called \emph{negatively
		$\mathbf t$-determined} if $N_{\mathbf a}=0$ whenever
	$a_j\ge t_j+1$ for some $j\in[n]$, and, for every $j\in[n]$,
	multiplication by $x_j$ is an isomorphism from
	$N_{\mathbf a}$ to $N_{\mathbf a+\mathbf e_j}$ whenever $a_j\le-1$.
	We denote the corresponding full subcategories by $\mathcal A^+_{\mathbf t}$
	and $\mathcal A^-_{\mathbf t}$, respectively.
	By \cite{BF}, if $M\in\mathcal A^+_{\mathbf t}$, then
	$H^i_{\mathfrak m}(M(-\mathbf1))\in\mathcal A^-_{\mathbf t}$.
	Thus we have the functor
	\[
	\Phi_i^{\mathbf t}:\mathcal A^+_{\mathbf t}\longrightarrow
	\mathcal A^-_{\mathbf t},\qquad
	M\longmapsto H^i_{\mathfrak m}(M(-\mathbf1)).
	\]
	Let $\mathcal H_i(\mathbf t)=\EssIm\Phi_i^{\mathbf t}$, where the
	essential image is the full subcategory of objects isomorphic to values
	of the functor.  We call $M\in\mathcal A^+_{\mathbf t}$ a \emph{source}
	of $L\in\mathcal H_i(\mathbf t)$ if $\Phi_i^{\mathbf t}(M)\cong L$.
	A source need not be unique.
	
	The questions considered in this paper are the following.  Which objects
	of $\mathcal A^-_{\mathbf t}$ occur as
	$H^i_{\mathfrak m}(M(-\mathbf1))$ for $M\in\mathcal A^+_{\mathbf t}$,
	and what is the structure of the resulting image category?  Can sources
	be chosen functorially, and can such a choice be exact or fully faithful?
	How does the categorical description determine attached primes?
	What are the projective and injective objects of these categories,
	and what are their Bass numbers and injective dimensions over $S$?
	
	We first describe $\mathcal H_i(\mathbf t)$ for $0\le i<n$.
	Let $P_{\mathbf t}=[\mathbf0,\mathbf t]$, and for
	$\mathbf a\in\mathbb N^n$ put
	\[
	\ell(\mathbf a)=|\supp(\mathbf a)|=\#\{j\in[n]:a_j>0\},
	\qquad
	U_q(\mathbf t)=\{\mathbf a\in P_{\mathbf t}:\ell(\mathbf a)\ge q\}.
	\]
	For a poset $P$, let $\Rep(P)$ be the category of covariant functors
	from $P$ to $\kk\text{-}\modcat$, with natural transformations as
	morphisms. Here, $\kk\text{-}\modcat$ is the
	category of finite-dimensional $\kk$-vector spaces and $\kk$-linear maps.
	Restriction to the box gives exact equivalences
	$\pi_{\mathbf t}:\mathcal A^+_{\mathbf t}\to\Rep(P_{\mathbf t})$ and
	$\theta_{\mathbf t}:\mathcal A^-_{\mathbf t}\to\Rep(P_{\mathbf t})$.
	Our first main result identifies the image categories and distinguishes
	the top degree from the non-top degrees.
	
	\begin{maintheorem}[Image categories]\label{main:classification}
		For $0\le i<n$, the category $\mathcal H_i(\mathbf t)$ consists precisely
		of the objects whose restrictions to $P_{\mathbf t}$ are supported on
		$U_{n-i}(\mathbf t)$.  It is a Serre subcategory of
		$\mathcal A^-_{\mathbf t}$, and restriction induces an exact equivalence
		\begin{equation}\label{eq:introduction-classification}
			\mathcal H_i(\mathbf t)\simeq\Rep(U_{n-i}(\mathbf t)).
		\end{equation}
		In top degree, $\theta_{\mathbf t}$ identifies $\mathcal H_n(\mathbf t)$
		with the full subcategory of second cosyzygies in $\Rep(P_{\mathbf t})$.
		This subcategory is extension-closed and is not abelian when
		$\mathbf t\ne\mathbf0$.
	\end{maintheorem}
	
	The non-top classification is Theorem~\ref{thm:finite-intrinsic}; the
	top-degree assertions are Theorem~\ref{thm:top-finite} and
	Proposition~\ref{prop:top-nonabelian}.  Here a Serre subcategory is a
	full subcategory closed under subobjects, quotients, and extensions.
	A second cosyzygy is the cokernel of a morphism between injective objects.
	
	The support description also relates the images in different degrees.
	The non-top image categories form an increasing chain
	\[
	\mathcal H_0(\mathbf t)\subseteq\mathcal H_1(\mathbf t)
	\subseteq\cdots\subseteq\mathcal H_{n-1}(\mathbf t),
	\]
	with all inclusions strict when $\mathbf t\ge\mathbf1$.
	The inclusions in this chain, as well as the embeddings into
	$\mathcal A^-_{\mathbf t}$, preserve projective objects and projective
	dimensions (Remark~\ref{rem:projectives-under-inclusions}).
	
	The classification also determines the homological dimensions of the
	non-top image categories.  For $\mathbf t\ge\mathbf1$, we obtain
	\[
	\operatorname{gl.dim}\mathcal H_i(\mathbf t)
	=\min\{n,i+\#\{j\in[n]:t_j\ge2\}\},\qquad 0\le i<n
	\]
	(Proposition~\ref{prop:image-global-dimension}).  In particular,
	$\operatorname{gl.dim}\mathcal H_i(\mathbf1)=i$.
	
	The chain also gives canonical torsion decompositions.  For $0\le s<n$, let
	$\mathcal B_s(\mathbf t)$ be the full subcategory of $\mathcal H_s(\mathbf t)$
	whose restrictions to the box are supported on the layer
	$\ell(\mathbf a)=n-s$.  For $0\le i\le n-2$,
	\[
	\bigl(\mathcal H_i(\mathbf t),\mathcal B_{i+1}(\mathbf t)\bigr)
	\]
	is a hereditary torsion pair in $\mathcal H_{i+1}(\mathbf t)$
	(Proposition~\ref{prop:adjacent-torsion-pair}).
	
	The relation with the top-degree image takes a different form:
	\[
	\mathcal H_{n-1}(\mathbf t)\cap\mathcal H_n(\mathbf t)=0,
	\]
	and moreover, every $L\in\mathcal A^-_{\mathbf t}$ fits into a short exact sequence
	\[
	0\longrightarrow L'\longrightarrow L\longrightarrow L''\longrightarrow0,
	\qquad L'\in\mathcal H_{n-1}(\mathbf t),\quad
	L''\in\mathcal H_n(\mathbf t).
	\]
	Thus these two subcategories together generate $\mathcal A^-_{\mathbf t}$
	under extensions (Propositions~\ref{prop:top-disjoint}
	and~\ref{prop:top-penultimate}).
	
	These two image categories can also be characterized directly by graded
	Matlis duality.  For $N\in\mathcal A^-_{\mathbf t}$, one has
	$N\in\mathcal H_{n-1}(\mathbf t)$ if and only if its graded Matlis dual $\Dgr N$
	is torsion, whereas $N\in\mathcal H_n(\mathbf t)$ if and only if the same dual is
	reflexive (Theorem~\ref{thm:reflexive-top-complement}).  Moreover,
	$\mathcal H_{n-1}(\mathbf t)\cup\mathcal H_n(\mathbf t)
	=\mathcal A^-_{\mathbf t}$ if and only if $\mathbf t=\mathbf0$
	(Corollary~\ref{cor:top-two-exhaustion}).
	
	We now turn to the commutative-algebra applications of the classification.
	Our second main result computes the attached primes of an arbitrary object
	in a non-top image category from its
	projective cover and characterizes all sets that can occur.
	For $F\subseteq[n]$, put $\mathfrak p_F=(x_j:j\in F)$, and write
	$\mathfrak p_{\mathbf a}=\mathfrak p_{\supp(\mathbf a)}$.
	For $\mathbf a\in U_{n-i}(\mathbf t)$, let
	$\mathcal P_{\mathbf a}^{\mathbf t}$ denote the indecomposable projective
	object of $\mathcal H_i(\mathbf t)$ corresponding under
	Theorem~\ref{main:classification} to the principal upper-interval
	representation at $\mathbf a$.
	
	\begin{maintheorem}[Attached primes and their realization]\label{main:attached-primes}
		Let $T=\{j\in[n]:t_j>0\}$.
		\begin{enumerate}[font=\upshape]
			\item For $0\le i<n$, let $L\in\mathcal H_i(\mathbf t)$ and let
			\[
			Q_0=\bigoplus_{\mathbf a\in U_{n-i}(\mathbf t)}
			(\mathcal P_{\mathbf a}^{\mathbf t})^{\mu_{\mathbf a}}
			\longrightarrow L
			\]
			be its projective cover in $\mathcal H_i(\mathbf t)$.  Then
			\[
			\operatorname{Att}_S L=\operatorname{Att}_S Q_0
			=\{\mathfrak p_{\mathbf a}:\mu_{\mathbf a}>0\}.
			\]
			\item For $0\le i<n$, put
			\[
			\mathcal C_i(\mathbf t)
			=\{\mathfrak p_F:F\subseteq T,\ |F|\ge n-i\}.
			\]
			Then
			\[
			\bigl\{\operatorname{Att}_S L:L\in\mathcal H_i(\mathbf t)\bigr\}
			=\mathcal P\bigl(\mathcal C_i(\mathbf t)\bigr),
			\]
			where $\mathcal P$ denotes the power set.  Every such set is realized
			by a projective object of $\mathcal H_i(\mathbf t)$.  In top degree,
			\[
			\bigl\{\operatorname{Att}_S L:L\in\mathcal H_n(\mathbf t)\bigr\}
			=\{\varnothing,\{(0)\}\}.
			\]
		\end{enumerate}
	\end{maintheorem}
	
	The two assertions are proved in
	Proposition~\ref{prop:attached-projective-cover} and
	Corollary~\ref{cor:possible-attached-primes}.  Thus the projective cover
	preserves the entire attached-prime set, and every subset of the allowed
	coordinate primes occurs.
	
	Graded Matlis duality also yields explicit Bass numbers for the projective
	and injective objects of the non-top image categories.  In particular, every nonzero injective object
	$I$ of $\mathcal H_i(\mathbf t)$ satisfies
	\[
	\operatorname{id}_S I=n-i\qquad(0\le i<n)
	\]
	(Corollary~\ref{cor:target-injective-dimensions}).
	
	We next turn to the functorial choice of sources.  The classification in Theorem~\ref{main:classification}
	does not by itself provide a functorial section of $\Phi_i^{\mathbf t}$.
	Our third main result gives exact sections in non-top degrees,
	obstructions to full faithfulness in the extremal non-top degrees,
	and a fully faithful section in top degree.
	
	\begin{maintheorem}[Functorial sections and full faithfulness]\label{main:sections}
		For every $0\le i<n$, there is an additive exact functor
		$s_i^{\mathbf t}:\mathcal H_i(\mathbf t)\to\mathcal A^+_{\mathbf t}$ such that
		\[
		\Phi_i^{\mathbf t}\circ s_i^{\mathbf t}
		\cong\operatorname{id}_{\mathcal H_i(\mathbf t)}.
		\]
		For every nonzero $L\in\mathcal H_i(\mathbf t)$, it satisfies
		\[
		\pd_{\mathcal A^+_{\mathbf t}}s_i^{\mathbf t}(L)
		=\pd_S s_i^{\mathbf t}(L)=n-i
		=\min\{\pd_S N\mid N\in\mathcal A^+_{\mathbf t},\
		\Phi_i^{\mathbf t}(N)\cong L\}.
		\]
		If $\mathcal H_i(\mathbf t)\ne0$, then $s_i^{\mathbf t}$ is faithful but
		not full.  Furthermore, if $n\ge2$ and $\mathbf t\ge\mathbf1$, then
		$\Phi_i^{\mathbf t}$ admits no fully faithful functorial section for
		$i\in\{1,n-1\}$, even without requiring exactness.
		In top degree, the assignment
		\[
		s_n^{\mathbf t}(N)=\underline{\Hom}_S(\Dgr N,S)
		\]
		defines an additive fully faithful functor
		$s_n^{\mathbf t}:\mathcal H_n(\mathbf t)\to\mathcal A^+_{\mathbf t}$
		with $\Phi_n^{\mathbf t}\circ s_n^{\mathbf t}
		\cong\operatorname{id}_{\mathcal H_n(\mathbf t)}$.
	\end{maintheorem}
	
	The construction and its properties are established in
	Theorem~\ref{thm:functorial-section} and Corollary~\ref{cor:section-optimality};
	the obstruction is Proposition~\ref{prop:no-fully-faithful-section},
	and the top-degree section is Theorem~\ref{thm:top-fully-faithful-section}.
	Thus an exact choice of sources need not realize all morphisms between
	the chosen sources.  In degree zero, the simpler shift section
	$L\mapsto L(\mathbf1)$ is fully faithful.
	
	Finally, we pass from a fixed determining box to the global categories.
	Let $\mathcal H_i=\bigcup_{\mathbf t\in\mathbb N^n}\mathcal H_i(\mathbf t)$
	and $U_q=\{\mathbf a\in\mathbb N^n:\ell(\mathbf a)\ge q\}$.
	The finite-box images are compatible with these global categories:
	\[
	\mathcal H_i(\mathbf t)=\mathcal H_i\cap\mathcal A^-_{\mathbf t}
	\qquad(0\le i\le n)
	\]
	(Proposition~\ref{prop:global-box-compatibility}).  Thus a globally
	realizable target in $\mathcal A^-_{\mathbf t}$ admits a source in
	$\mathcal A^+_{\mathbf t}$.
	
	There is also a criterion for the original source to lie in the box.
	Put $\mathcal A^+=\bigcup_{\mathbf u\in\mathbb N^n}\mathcal A^+_{\mathbf u}$.
	For $M\in\mathcal A^+$,
	\[
	M\in\mathcal A^+_{\mathbf t}
	\quad\Longleftrightarrow\quad
	H^i_{\mathfrak m}(M(-\mathbf1))\in\mathcal A^-_{\mathbf t}
	\text{ for every }0\le i\le n
	\]
	(Proposition~\ref{prop:determination-by-local-cohomology}).
	The first statement permits a new source to be chosen within the box;
	the second detects whether the original source itself belongs to that box.
	
	The global image categories themselves retain a representation-theoretic
	description.  Restriction to nonnegative degrees gives an equivalence
	$\theta$ from $\mathcal A^-:=\bigcup_{\mathbf t\in\mathbb N^n}\mathcal A^-_{\mathbf t}$ to
	$\Rep_{\mathrm{fs}}(\mathbb N^n)$ and induces the global description
	\[
	\mathcal H_i\simeq\Rep_{\mathrm{fs}}(U_{n-i}),\qquad 0\le i<n,
	\]
	where $\Rep_{\mathrm{fs}}$ denotes representations with finite support.
	For $0\le i<n$, the category $\mathcal H_i$ has enough injectives
	and no nonzero projective objects.  Its indecomposable injectives
	correspond to principal lower-interval representations
	$I_{\mathbf a}^{(q)}$, where $q=n-i$.  The category $\mathcal H_n$
	corresponds under $\theta$ to the full subcategory of finite-support
	second cosyzygies.  Moreover, $\mathcal H_i\cap\mathcal H_n=0$ for $i<n$,
	and $\mathcal H_{n-1}$ and $\mathcal H_n$ together generate
	$\mathcal A^-$ under extensions.
	
	These results lead to questions concerning analogous image
	categories and source constructions for complete-intersection monomial
	ideals and for general monomial ideals.
	
	We use the following homological notation throughout the paper.
	For a finitely generated $\mathbb Z^n$-graded $S$-module $M$, write $\pd_S M$ for its
	ordinary (ungraded) projective dimension over $S$, $\dim_S M$ for its Krull dimension, and
	$\depth_S M=\depth_{S_{\mathfrak m}}M_{\mathfrak m}$.  A nonzero such
	module is Cohen--Macaulay if its depth equals its dimension.
	If $\cdots\to F_1\to F_0\to M\to0$ is a minimal multigraded free
	resolution, set $\Omega_S^0M=M$ and
	$\Omega_S^rM=\operatorname{im}(F_r\to F_{r-1})$ for $r\ge1$;
	$\Omega$ without a subscript also denotes syzygy over $S$.
	We use $\underline{\operatorname{Hom}}_S(M,N)$ and
	$\underline{\operatorname{Ext}}_S^r(M,N)$ for graded Hom and Ext, allowing
	homogeneous maps of all multidegrees.  Hom and End in the graded
	categories refer to degree-zero maps.  Put $\omega_S=S(-\mathbf1)$.
	The multigraded and total Betti numbers are denoted by
	$\beta_{r,\mathbf u}^S(M)=\dim_\kk\operatorname{Tor}_r^S(\kk,M)_{\mathbf u}$
	and $\beta_r^S(M)=\sum_{\mathbf u\in\mathbb Z^n}\beta_{r,\mathbf u}^S(M)$, respectively.
	We omit the superscript $S$ when the ring is understood.
	For an abelian category $\mathcal C$, we write
	$\operatorname{gl.dim}\mathcal C$ for its global dimension and
	$\proj\mathcal C$ for its full subcategory of projective objects.
	Unless otherwise stated, $0\le i<n$ and $q=n-i$.
	
	The paper is organized as follows.  In Section~\ref{sec:determined}, we
	recall determined modules and their finite-box descriptions.
	Section~\ref{sec:homological} establishes the homological results needed for
	the classification, which is proved in Section~\ref{sec:finitebox}.
	Section~\ref{sec:attached-bass} gives the applications to attached primes
	and Bass numbers.
	In Section~\ref{sec:functorial-section}, we construct exact functorial
	sections in non-top degrees, establish their homological optimality,
	prove obstructions to full faithfulness, and construct a fully faithful
	section in top degree.  In
	Section~\ref{sec:global-sources}, we study the global finite-support
	categories.
	Finally, Section~\ref{sec:open-problems}
	presents open questions concerning other monomial supports.
	
	\section{Determined modules and finite-box models}\label{sec:determined}
	In this section, we recall positively and negatively $\mathbf t$-determined
	modules and their descriptions as representations of the finite box
	$P_{\mathbf t}$.  We establish the global-dimension formula for products of
	finite posets and compare projective and injective dimensions for different
	determining bounds.  We also describe the global category $\mathcal A^-$
	by finite-support representations of $\mathbb N^n$ and introduce the
	local-cohomology functors $\Phi_i^{\mathbf t}$ and their image categories.
	
	Let $\mathsf{Gr}_{\mathbb Z^n}(S)$ denote the category of all
	$\mathbb Z^n$-graded $S$-modules with degree-zero $S$-linear maps.
	A $\mathbb Z^n$-graded $S$-module is called \emph{locally finite} if
	each of its graded components is finite-dimensional over $\kk$.
	We denote by $\LFgr$ the full subcategory of
	$\mathsf{Gr}_{\mathbb Z^n}(S)$ consisting of these modules.
	For $M\in\LFgr$, let
	\[
	\Dgr M=\bigoplus_{\mathbf a\in \mathbb Z^n}\Hom_\kk(M_{-\mathbf a},\kk)
	\]
	denote its graded $\kk$-dual, with the natural $\mathbb Z^n$-graded
	$S$-module structure. This graded Matlis duality gives an exact contravariant
	equivalence between finitely generated and Artinian $\mathbb Z^n$-graded
	$S$-modules; see \cite[Theorem~3.6.17]{BH} and
	\cite[Section~1.2]{Miller}.  We use
	$D=\Hom_\kk(-,\kk)$ for duality on finite-dimensional vector spaces.
	
	Let $\mathcal A^+$ denote the full subcategory of finitely generated
	modules $M\in\LFgr$ such that $M_{\mathbf a}=0$ whenever
	$\mathbf a\notin\mathbb N^n$.  Fix $\mathbf t\in\mathbb N^n$.
	Following \cite{BF,Miller}, a module $M\in\mathcal A^+$ is said to be
	\emph{positively $\mathbf t$-determined} if multiplication induces an
	isomorphism
	\[
	x_j:M_{\mathbf a}\xrightarrow{\sim}M_{\mathbf a+\mathbf e_j}
	\]
	for every $\mathbf a\in\mathbb N^n$ and every $j\in[n]$ such that
	$a_j\ge t_j$.
	We denote the full subcategory of these modules by
	$\mathcal A^+_{\mathbf t}$.  It is closed under kernels, cokernels and
	extensions, and hence is abelian.  Put
	$P_{\mathbf t}=[\mathbf0,\mathbf t]
	=\{\mathbf a\in\mathbb N^n:\mathbf0\le\mathbf a\le\mathbf t\}$,
	with the order induced from $\mathbb N^n$.  We recall its relation to
	positively determined modules.
	
	\begin{proposition}\label{prop:positive-box}
		For $M\in\mathcal A^+_{\mathbf t}$, let $\pi_{\mathbf t}(M)$ be the
		representation of $P_{\mathbf t}$ with
		$\pi_{\mathbf t}(M)_{\mathbf a}=M_{\mathbf a}$ and structure map
		$M_{\mathbf a}\to M_{\mathbf b}$ given by multiplication by
		$x^{\mathbf b-\mathbf a}$ whenever $\mathbf a\le\mathbf b$.
		Then $\pi_{\mathbf t}$ is an exact equivalence
		\[
		\pi_{\mathbf t}:\mathcal A^+_{\mathbf t}\xrightarrow{\sim}\Rep(P_{\mathbf t}).
		\]
		Moreover, $\mathcal A^+=\bigcup_{\mathbf t\in\mathbb N^n}\mathcal A^+_{\mathbf t}$.
	\end{proposition}
	\begin{proof}
		Let $V\in\Rep(P_{\mathbf t})$.  For $\mathbf a\in\mathbb N^n$,
		set $M_{\mathbf a}=V_{\mathbf a\wedge\mathbf t}$, where
		$\mathbf a\wedge\mathbf t$ is the coordinatewise minimum, and set
		$M_{\mathbf a}=0$ outside $\mathbb N^n$.  Use the structure maps of $V$
		inside the box and identity maps when a coordinate has reached $t_j$.
		Then $M$ is positively $\mathbf t$-determined, and this construction
		gives a quasi-inverse to $\pi_{\mathbf t}$.  Exactness is checked in
		each degree.  Finally, every $M\in\mathcal A^+$ is positively
		$\mathbf t$-determined for $\mathbf t$ sufficiently large
		\cite[Proposition~2.5]{Miller}, proving the last assertion.
	\end{proof}
	
	For a finitely generated multigraded $S$-module $N$, let $\pd_S N$
	denote its ordinary projective dimension.  By
	\cite[Corollary~A.I.2.7]{NVO}, it equals the graded projective dimension
	and hence can be computed from a minimal multigraded free resolution.
	
	\begin{proposition}
		\label{prop:box-homological-dimensions}
		\begin{enumerate}[label=\textup{(\roman*)},leftmargin=*]
			\item Let $P$ and $Q$ be nonempty finite posets.  Then
			\[
			\operatorname{gl.dim}\Rep(P\times Q)
			=\operatorname{gl.dim}\Rep(P)+\operatorname{gl.dim}\Rep(Q).
			\]
			In particular, for $\mathbf t\in\mathbb N^n$,
			\[
			\operatorname{gl.dim}\mathcal A^+_{\mathbf t}
			=\operatorname{gl.dim}\Rep([\mathbf0,\mathbf t])
			=\#\{j:1\le j\le n,\ t_j>0\}.
			\]
			\item Let $N\in\mathcal A^+_{\mathbf t}\cap\mathcal A^+_{\mathbf u}$.  Then
			\[
			\pd_{\mathcal A^+_{\mathbf t}}N
			=\pd_{\mathcal A^+_{\mathbf u}}N
			=\pd_S N,
			\qquad
			\operatorname{id}_{\mathcal A^+_{\mathbf t}}N
			=\operatorname{id}_{\mathcal A^+_{\mathbf u}}N.
			\]
		\end{enumerate}
	\end{proposition}
	
	\begin{proof}
		\textup{(i)} For $V\in\Rep(P)$ and $W\in\Rep(Q)$, let
		$V\boxtimes W\in\Rep(P\times Q)$ be given by
		$(V\boxtimes W)_{(p,q)}=V_p\otimes_\kk W_q$, with the induced structure maps.
		Write $S_p$ and $E_p$ for the vertex simple and principal projective at
		$p\in P$, and use the same notation for $Q$ and $P\times Q$.
		Every simple representation of $P\times Q$ has the form
		$S_{(p,q)}=S_p\boxtimes S_q$, and
		$E_p\boxtimes E_q\cong E_{(p,q)}$.
		Choose minimal projective resolutions $F_\bullet\to S_p$ and
		$G_\bullet\to S_q$.  Tensoring over $\kk$ and taking the total complex
		gives a projective resolution $C_\bullet\to S_{(p,q)}$, where
		\[
		C_j=\bigoplus_{r+s=j}F_r\boxtimes G_s,
		\qquad
		d_C|_{F_r\boxtimes G_s}
		=d_F\boxtimes1+(-1)^r1\boxtimes d_G.
		\]
		
		We verify that $C_\bullet$ is minimal.  Recall that a projective
		resolution is minimal if and only if the image of each differential
		is contained in the radical of the following projective object.
		The principal projective $E_a$ has unique simple quotient $S_a$, so
		$\operatorname{rad}E_a=\ker(E_a\twoheadrightarrow S_a)$.
		Since $\operatorname{Hom}(E_a,S_b)=0$ for $a\ne b$, every map
		$E_a\to E_b$ between distinct principal projectives has image in
		$\operatorname{rad}E_b$.  On the other hand,
		$\operatorname{End}(E_a)=\kk$, so any nonzero map between two copies
		of $E_a$ is an isomorphism.  Minimality of $F_\bullet$ and $G_\bullet$
		therefore implies that their differential components between copies
		of the same principal projective are zero.  Thus every nonzero
		matrix component of $d_C$ has one of the forms
		\[
		E_{(a,b)}\longrightarrow E_{(a',b)}\quad(a\ne a'),
		\qquad
		E_{(a,b)}\longrightarrow E_{(a,b')}\quad(b\ne b').
		\]
		In either case its image lies in the radical of the target.
		Consequently $\operatorname{im}d_{C,j}\subseteq
		\operatorname{rad}C_{j-1}$ for every $j\ge1$, proving minimality.
		
		If $F_\bullet$ and $G_\bullet$ have lengths $m$ and $l$, respectively,
		then the last nonzero term of $C_\bullet$ is
		$C_{m+l}=F_m\boxtimes G_l$.  Hence
		\[
		\pd_{\Rep(P\times Q)}S_{(p,q)}
		=\pd_{\Rep(P)}S_p+\pd_{\Rep(Q)}S_q.
		\]
		Taking the maximum over $p\in P$ and $q\in Q$ proves the first formula,
		since global dimension is the maximum of the projective dimensions of
		the vertex simples.  Finally, Proposition~\ref{prop:positive-box} and
		$P_{\mathbf t}=\prod_{j=1}^n[0,t_j]$ give
		\[
		\operatorname{gl.dim}\mathcal A^+_{\mathbf t}
		=\sum_{j=1}^n\operatorname{gl.dim}\Rep([0,t_j])
		=\#\{j:1\le j\le n,\ t_j>0\},
		\]
		since $\Rep([0,t_j])$ is semisimple when $t_j=0$ and hereditary but
		not semisimple otherwise.
		
		\textup{(ii)} By Proposition~\ref{prop:positive-box} and \cite[Section~2.5]{Ladkani},
		the indecomposable projectives
		of $\mathcal A^+_{\mathbf t}$ are $S(-\mathbf b)$ with
		$\mathbf0\le\mathbf b\le\mathbf t$.  The minimal generators of a positively
		$\mathbf t$-determined module have degrees in this box.  Thus its minimal free cover and the kernel of this cover both belong
		to $\mathcal A^+_{\mathbf t}$.  Repeating this argument shows that the
		minimal multigraded free resolution of $N$ is a projective resolution
		in $\mathcal A^+_{\mathbf t}$.  Conversely, every projective object
		is a finite direct sum of the displayed free modules.  Hence
		$\pd_{\mathcal A^+_{\mathbf t}}N=\pd_S N$, and projective dimension
		is independent of the determining bound.
		
		We next consider injective dimension.  Suppose first that
		$\mathbf t\le\mathbf u$.  The inclusion
		$\mathcal A^+_{\mathbf t}\hookrightarrow\mathcal A^+_{\mathbf u}$ is exact
		and full.  The indecomposable injective corresponding to $\mathbf b\le\mathbf t$
		is the module
		$J_{\mathbf b}^{\mathbf t}=S/(x_j^{b_j+1}:1\le j\le n,\ b_j<t_j)$.
		Define $\mathbf c$ by $c_j=b_j$ when $b_j<t_j$ and $c_j=u_j$
		otherwise.  Then $J_{\mathbf b}^{\mathbf t}=J_{\mathbf c}^{\mathbf u}$.
		Thus the inclusion preserves injective objects, and an injective
		resolution in $\mathcal A^+_{\mathbf t}$ is also one in
		$\mathcal A^+_{\mathbf u}$.  Since the inclusion is exact and fully
		faithful and preserves injective objects, we obtain
		\[
		\Ext^r_{\mathcal A^+_{\mathbf t}}(X,N)
		\cong\Ext^r_{\mathcal A^+_{\mathbf u}}(X,N)
		\]
		for $X,N\in\mathcal A^+_{\mathbf t}$.  The preservation of injective
		resolutions gives
		$\operatorname{id}_{\mathcal A^+_{\mathbf u}}N
		\le\operatorname{id}_{\mathcal A^+_{\mathbf t}}N$, while the Ext
		isomorphisms give the reverse inequality.  For arbitrary $\mathbf t$ and $\mathbf u$, let $\mathbf v$ be their
		coordinatewise maximum. Applying the preceding argument to
		$\mathbf t\le\mathbf v$ and $\mathbf u\le\mathbf v$, we obtain
		$\operatorname{id}_{\mathcal A^+_{\mathbf t}}N
		=\operatorname{id}_{\mathcal A^+_{\mathbf v}}N
		=\operatorname{id}_{\mathcal A^+_{\mathbf u}}N$.
	\end{proof}
	
	\begin{remark}
		Unlike projective dimension, injective dimension depends on the category:
		$S=J_{\mathbf t}^{\mathbf t}$ is injective in $\mathcal A^+_{\mathbf t}$,
		so $\operatorname{id}_{\mathcal A^+_{\mathbf t}}S=0$, whereas
		$\operatorname{id}_S S=n$ for $n\ge1$.
	\end{remark}
	
	We now recall the corresponding notion for Artinian modules.  Following
	Brun--Fl\o ystad \cite[Remark~2.9]{BF}, an Artinian $\mathbb Z^n$-graded
	$S$-module $N$ is called \emph{negatively $\mathbf t$-determined} if,
	for all $\mathbf a\in\mathbb Z^n$ and $j\in[n]$,
	\[
	N_{\mathbf a}=0\quad\text{when }a_j\ge t_j+1,
	\qquad
	x_j:N_{\mathbf a}\xrightarrow{\sim}N_{\mathbf a+\mathbf e_j}
	\quad(a_j\le-1).
	\]
	Let $\mathcal A^-_{\mathbf t}$ denote the full subcategory of these
	modules.  It is abelian, and its inclusion in the category of Artinian
	graded modules is exact.  Put
	$\mathcal A^-=\bigcup_{\mathbf t\in\mathbb N^n}\mathcal A^-_{\mathbf t}$.
	Both $\mathcal A^-_{\mathbf t}$ and $\mathcal A^-$ are full subcategories
	of $\LFgr$.
	
	\begin{proposition}\label{prop:negative-box}
		Restriction to nonnegative degrees induces an exact equivalence
		\[
		\theta_{\mathbf t}:\mathcal A^-_{\mathbf t}
		\xrightarrow{\sim}\Rep(P_{\mathbf t}).
		\]
	\end{proposition}
	
	\begin{proof}
		Let $N\in\mathcal A^-_{\mathbf t}$.  For $\mathbf b\in\mathbb Z^n$, put
		$\mathbf b^+=(\max\{b_1,0\},\ldots,\max\{b_n,0\})$.  If
		$\mathbf b\le\mathbf t$, negative determination gives a canonical isomorphism
		$N_{\mathbf b}\cong N_{\mathbf b^+}$.  It follows that a morphism is uniquely determined by its restriction
		to $P_{\mathbf t}$, and every morphism on the box extends.  Thus
		restriction is full and faithful.  To prove essential surjectivity,
		let $V\in\Rep(P_{\mathbf t})$ and define
		\[
		(\sigma_{\mathbf t}V)_{\mathbf b}=
		\begin{cases}
			V_{\mathbf b^+},&\mathbf b\le\mathbf t,\\
			0,&\text{otherwise},
		\end{cases}
		\]
		using identity maps in negative directions and the structure maps of $V$ in
		the box.  The resulting module satisfies the negative determination conditions.
		Its graded dual is generated in the finite box $[-\mathbf t,\mathbf0]$
		and is therefore finitely generated.  Graded Matlis duality shows that
		$\sigma_{\mathbf t}V$ is Artinian.  Hence $\sigma_{\mathbf t}$ is a
		quasi-inverse to restriction.  Both functors are exact in each degree.
	\end{proof}
	
	Passing to the union of the finite boxes gives the following result.
	Here $\Rep_{\mathrm{fs}}(\mathbb N^n)$ denotes the category of
	representations with finite support.
	
	\begin{proposition}\label{prop:global-fs}
		Restriction to nonnegative degrees induces an exact equivalence
		\[
		\theta:\mathcal A^-\xrightarrow{\sim}\Rep_{\mathrm{fs}}(\mathbb N^n).
		\]
	\end{proposition}
	
	\begin{proof}
		Every object of $\mathcal A^-$ belongs to some $\mathcal A^-_{\mathbf t}$,
		and every finite-support representation is supported in a box
		$P_{\mathbf t}$.  The reconstruction in
		Proposition~\ref{prop:negative-box} is compatible with enlarging the
		box.  The assertion follows.
	\end{proof}
	
	Finally, we introduce the local-cohomology functors to be studied.
	For a $\mathbb Z^n$-graded module $M$, recall that our shift convention is
	$M(-\mathbf c)_{\mathbf a}=M_{\mathbf a-\mathbf c}$.
	By \cite[Proposition~2.8 and Remark~2.9]{BF}, if
	$M\in\mathcal A^+_{\mathbf t}$ then
	$H^i_{\mathfrak m}(M(-\mathbf1))\in\mathcal A^-_{\mathbf t}$.  Thus we have an additive  functor
	\[
	\Phi_i^{\mathbf t}:\mathcal A^+_{\mathbf t}\longrightarrow
	\mathcal A^-_{\mathbf t},\qquad
	M\longmapsto H^i_{\mathfrak m}(M(-\mathbf1)).
	\]
	We write
	\[
	\mathcal H_i(\mathbf t)=\EssIm\Phi_i^{\mathbf t},
	\qquad
	\mathcal H_i=\bigcup_{\mathbf t\in\mathbb N^n}\mathcal H_i(\mathbf t).
	\]
Clearly, if $\mathbf t\le\mathbf u$, then
	$\mathcal H_i(\mathbf t)\subseteq\mathcal H_i(\mathbf u)$.
	
	\section{Incidence algebras, Auslander regularity, and grade}\label{sec:homological}
	The finite-box equivalences (stated in Section 2) allow us to study local cohomology through
	modules over an incidence algebra.  In this section, we first express
	the local-cohomology functors in terms of dual Ext functors.  We then use the Auslander
	condition to characterize the essential images of positive-degree Ext
	functors, preparing the classification in Section~\ref{sec:finitebox}.
	
	We use cohomological grading for complexes.  In particular, a projective
	resolution $P_\bullet\to M$ is placed in nonpositive degrees, with $P_r$
	in degree $-r$.  For a cochain complex $X$, our shift convention is
	$(X[r])^m=X^{m+r}$.  If its terms are finite-dimensional $\kk$-vector
	spaces, the dual complex is graded by $(DX)^m=D(X^{-m})$.
	Thus $\Hom(P_\bullet,A)$ is concentrated in nonnegative degrees,
	and applying $D$ reverses these degrees.
	
	\subsection{The Nakayama--Ext translation}
	
	Modules over finite-dimensional $\kk$-algebras are understood to be
	finite-dimensional.  Let $P$ be a finite poset.  Its
	\emph{incidence algebra} $A_P$ over $\kk$
	is the $\kk$-vector space with basis
	\[
	\{e_{x,y}\mid x,y\in P,\ y\le x\},
	\]
	whose multiplication is determined by
	\[
	e_{x,y}e_{z,w}=\delta_{y,z}e_{x,w},\qquad\text{where }\delta_{y,z}\text{ is the Kronecker delta}.
	\]
	Its identity element is $\sum_{x\in P}e_{x,x}$.
	A left $A_P$-module $M$ determines a covariant representation of $P$:
	at $x$ we take $M_x=e_{x,x}M$, and for $y\le x$ the structure map
	$M_y\to M_x$ is given by $m\mapsto e_{x,y}m$.
	This construction yields the equivalences
	\[
	A_P\text{-mod}\simeq\Rep(P),
	\qquad
	\modcat\text{-}A_P\simeq\Rep(P^{\mathrm{op}}).
	\]
	
	Let $A$ be a finite-dimensional $\kk$-algebra.  The Nakayama functor
	on finitely generated projective left $A$-modules is
	$\nu_A=D\circ\Hom_A(-,A)$; see \cite{ARS}.  Since
	$\nu_A\cong DA\otimes_A-$ on projectives, applying it to a projective
	resolution yields the derived Nakayama functor $DA\otimes_A^{\mathbf L}-$;
	see also \cite[Section~3.6]{BF}.
	
	Put $\Lambda_{\mathbf t}=A_{P_{\mathbf t}}$ and use the above
	equivalences for $P=P_{\mathbf t}$ throughout.  For
	$\mathbf a\in P_{\mathbf t}$, let $S_{\mathbf a}$ denote the vertex simple in
	$\Rep(P_{\mathbf t})$, and let $R_{\mathbf a}$ denote the simple right
	$\Lambda_{\mathbf t}$-module corresponding to the vertex simple at
	$\mathbf a$ in $\Rep(P_{\mathbf t}^{\mathrm{op}})$.
	
	By Propositions~\ref{prop:positive-box} and~\ref{prop:negative-box},
	we have exact equivalences
	\[
	\pi_{\mathbf t}:\mathcal A^+_{\mathbf t}\xrightarrow{\sim}\Rep(P_{\mathbf t}),
	\qquad
	\theta_{\mathbf t}:\mathcal A^-_{\mathbf t}\xrightarrow{\sim}\Rep(P_{\mathbf t}).
	\]
	We will also use the order isomorphism
	\[
	\tau_{\mathbf t}:P_{\mathbf t}\longrightarrow P_{\mathbf t}^{\mathrm{op}},
	\qquad
	\tau_{\mathbf t}(\mathbf a)=\mathbf t-\mathbf a.
	\]
	Precomposition gives mutually inverse functors
	\[
	\tau_{\mathbf t}^*: \Rep(P_{\mathbf t}^{\mathrm{op}})
	\longrightarrow \Rep(P_{\mathbf t}),
	\qquad
	\tau_{\mathbf t}^*: \Rep(P_{\mathbf t})
	\longrightarrow \Rep(P_{\mathbf t}^{\mathrm{op}}),
	\]
	both given on vertices by
	$$(\tau_{\mathbf t}^*V)_{\mathbf a}=V_{\mathbf t-\mathbf a}.$$
	We use $\tau_{\mathbf t}^*$ for both functors.  Recall that
	$D=\Hom_\kk(-,\kk)$.  Applying $D$ at each vertex gives exact
	contravariant equivalences, which we write as
	\[
	D:(\modcat\text{-}\Lambda_{\mathbf t})^{\mathrm{op}}
	\xrightarrow{\sim}
	\Lambda_{\mathbf t}\text{-mod},
	\qquad
	D:\Rep(P_{\mathbf t}^{\mathrm{op}})^{\mathrm{op}}
	\xrightarrow{\sim}
	\Rep(P_{\mathbf t}).
	\]
	Here $D$ is vertexwise duality; the graded dual of an $S$-module is
	denoted by $\Dgr$, as in Section~\ref{sec:determined}.  Since
	$\tau_{\mathbf t}^2=\mathrm{id}$ and duality commutes with reindexing,
	we obtain a natural isomorphism
	\[
	\tau_{\mathbf t}^*\circ D\circ\tau_{\mathbf t}^*\cong D.
	\]
	This identity will remove the two reindexing functors in the
	Brun--Fl\o ystad formula below.  Note also that
	\[
	D(X)_{\mathbf a}=D(X_{\mathbf a}),
	\]
	so $D$ preserves vertex support.
	
	Let
	\[
	\mathcal T_i(\mathbf t):=\theta_{\mathbf t}(\mathcal H_i(\mathbf t))
	\subseteq\Rep(P_{\mathbf t}).
	\]
	Since $\theta_{\mathbf t}:\mathcal A^-_{\mathbf t}\xrightarrow{\sim}\Rep(P_{\mathbf t})$
	is an exact equivalence, it restricts to an equivalence
	\[
	\theta_{\mathbf t}|_{\mathcal H_i(\mathbf t)}:
	\mathcal H_i(\mathbf t)\xrightarrow{\sim}\mathcal T_i(\mathbf t).
	\]
	The next proposition identifies $\mathcal T_i(\mathbf t)$ with the
	essential image of a dual Ext functor.  It follows from the Nakayama
	functor description of Brun--Fl\o ystad.
	
	\begin{proposition}\label{prop:naka-ext}
		For $0\le i\le n$,
		there is a natural isomorphism of functors
		\[
		\theta_{\mathbf t}\circ\Phi_i^{\mathbf t}
		\;\cong\;
		D\circ\Ext^{n-i}_{\Lambda_{\mathbf t}}(-,\Lambda_{\mathbf t})
		\circ\pi_{\mathbf t}:
		\mathcal A^+_{\mathbf t}\longrightarrow\Rep(P_{\mathbf t}).
		\]
		Equivalently, for every $M\in\mathcal A^+_{\mathbf t}$,
		\[
		(\theta_{\mathbf t}\circ H^i_{\mathfrak m})(M(-\mathbf1))
		\cong
		\bigl(D\circ\Ext^{n-i}_{\Lambda_{\mathbf t}}(-,\Lambda_{\mathbf t})
		\circ\pi_{\mathbf t}\bigr)(M),
		\]
		and hence
		\[
		\mathcal T_i(\mathbf t)
		=
		\EssIm\left(
		D\circ\Ext^{n-i}_{\Lambda_{\mathbf t}}(-,\Lambda_{\mathbf t})
		:\Lambda_{\mathbf t}\text{-mod}\longrightarrow\Rep(P_{\mathbf t})
		\right).
		\]
	\end{proposition}
	
	\begin{proof}
		Let $M\in\mathcal A^+_{\mathbf t}$ and put $X=\pi_{\mathbf t}(M).$
		Choose a projective resolution $F_\bullet\to M$ in
		$\mathcal A^+_{\mathbf t}$ and put $P_\bullet=\pi_{\mathbf t}(F_\bullet).$
		Then $P_\bullet\to X$ is a projective resolution over
		$\Lambda_{\mathbf t}$.
		
		Let $\mathcal N_{\mathbf t}^{\mathbf1}$ be the first full Nakayama
		functor of Brun--Fl\o ystad \cite[Section~2.3]{BF}.  It acts on
		positively $\mathbf t$-determined complexes; the superscript
		$\mathbf1$ indicates one application in each coordinate.  Under the
		equivalence $\pi_{\mathbf t}$, \cite[Proposition~3.9]{BF} gives a
		natural quasi-isomorphism
		\[
		\bigl(\tau_{\mathbf t}^*\circ D\circ\tau_{\mathbf t}^*
		\circ\Hom_{\Lambda_{\mathbf t}}(-,\Lambda_{\mathbf t})\bigr)(P_\bullet)
		\simeq
		\bigl((\pi_{\mathbf t}\circ\mathcal N_{\mathbf t}^{\mathbf1})(M)\bigr)[n].
		\]
		To check the shift, recall that their convention is homological, with
		$(TC)_r=C_{r-1}$.  Writing $C^r=C_{-r}$ gives $(TC)^r=C^{r+1}$, so
		$T^n$ becomes $[n]$ in our convention.  Since $\tau_{\mathbf t}^*\circ D\circ\tau_{\mathbf t}^*\cong D$, this becomes
		\[
		\bigl(D\circ\Hom_{\Lambda_{\mathbf t}}
		(-,\Lambda_{\mathbf t})\bigr)(P_\bullet)
		\simeq
		\bigl((\pi_{\mathbf t}\circ\mathcal N_{\mathbf t}^{\mathbf1})(M)\bigr)[n].
		\]
		On the other hand, \cite[Proposition~2.7 and Definition~2.10]{BF}
		identify local cohomology with Nakayama cohomology on the box.  Hence
		\[
		(\theta_{\mathbf t}\circ H^i_{\mathfrak m})(M(-\mathbf1))
		\cong
		(\pi_{\mathbf t}\circ H^i\circ\mathcal N_{\mathbf t}^{\mathbf1})(M).
		\]
		Since $\pi_{\mathbf t}$ is exact,
		\[
		(\pi_{\mathbf t}\circ H^i\circ\mathcal N_{\mathbf t}^{\mathbf1})(M)
		\cong
		(H^i\circ\pi_{\mathbf t}\circ\mathcal N_{\mathbf t}^{\mathbf1})(M).
		\]
		Therefore
		\[
		\begin{aligned}
			(\theta_{\mathbf t}\circ H^i_{\mathfrak m})(M(-\mathbf1))
			&\cong
			H^{i-n}\!\left(
			\bigl((\pi_{\mathbf t}\circ\mathcal N_{\mathbf t}^{\mathbf1})(M)\bigr)[n]
			\right)\\
			&\cong
			\bigl(H^{i-n}\circ D\circ\Hom_{\Lambda_{\mathbf t}}
			(-,\Lambda_{\mathbf t})\bigr)(P_\bullet)\\
			&\cong
			\bigl(D\circ\Ext_{\Lambda_{\mathbf t}}^{\,n-i}
			(-,\Lambda_{\mathbf t})\bigr)(X).
		\end{aligned}
		\]
		Indeed, for a finite-dimensional cochain complex $C$,
		$(H^r\circ D)(C)\cong(D\circ H^{-r})(C)$, while
		\[
		\bigl(H^q\circ\Hom_{\Lambda_{\mathbf t}}(-,\Lambda_{\mathbf t})\bigr)(P_\bullet)
		\cong\Ext^q_{\Lambda_{\mathbf t}}(X,\Lambda_{\mathbf t}).
		\]
		The comparison maps are natural in $M$, and different choices of
		projective resolution are related by chain homotopy.  Thus the displayed
		isomorphism is natural.  Since $\pi_{\mathbf t}$ is essentially
		surjective, the assertion about essential images follows.
	\end{proof}
	
	\subsection{Auslander regular algebras}
	
	To determine the essential images of the Ext functors above, we use
	the Auslander condition.  We recall it in the form needed here.  Let $A$ be
	a finite-dimensional $\kk$-algebra, and let
	\[
	0\longrightarrow A_A\longrightarrow I^0\longrightarrow I^1\longrightarrow I^2\longrightarrow\cdots
	\]
	be a minimal injective coresolution of the right regular module.
	
	\begin{definition}
		For $k\ge1$, the algebra $A$ is \emph{$k$-Gorenstein} if
		\[
		\operatorname{pd}_{A^{\mathrm{op}}} I^r\le r
		\qquad (0\le r<k).
		\]
		It satisfies the \emph{Auslander condition} if it is $k$-Gorenstein for every $k$.  It is \emph{Auslander regular} if it satisfies the Auslander condition and has finite global dimension.
	\end{definition}
	
	This is the convention used in \cite{IM}.  For a right $A$-module $X$,
	its \emph{grade} is defined by
	$$
	\grade_{A^{\mathrm{op}}} X
	=
	\inf\{r\ge0\mid \Ext^r_{A^{\mathrm{op}}}(X,A)\ne0\},
	$$
	with the convention $\inf\varnothing=\infty$; in particular,
	$\grade_{A^{\mathrm{op}}}0=\infty$.
	By Auslander's theorem \cite[Theorem~3.7]{FGR}, $A$ satisfies the
	Auslander condition if and only if, for every left $A$-module $M$,
	every $q\ge1$ and every right submodule
	$$
	X\subseteq\Ext_A^q(M,A),
	$$
	one has $\grade_{A^{\mathrm{op}}} X\ge q.$

	\subsection{A realization lemma for Ext}
	
	The Auslander condition gives a lower bound for the grade of
	$\Ext_A^q(M,A)$.  The next lemma shows that every right module
	satisfying this bound occurs in this way.
	
	\begin{lemma}\label{lem:realization}
		Let $A$ be a finite-dimensional algebra satisfying the Auslander condition and let $q\ge1$.  Then
		\[
		\EssIm\left(
		\Ext_A^q(-,A):(A\text{-}\modcat)^{\mathrm{op}}\longrightarrow\modcat\text{-}A
		\right)
		=
		\{X\in\modcat\text{-}A\mid \grade_{A^{\mathrm{op}}} X\ge q\}.
		\]
	\end{lemma}
	
	\begin{proof}
		The inclusion from left to right follows from the Auslander condition.
		
		Conversely, let $X$ be a right $A$-module with $\grade_{A^{\mathrm{op}}} X\ge q$ and choose a projective resolution
		\[
		\cdots\longrightarrow P_q\longrightarrow P_{q-1}\longrightarrow\cdots\longrightarrow P_0\longrightarrow X\longrightarrow0.
		\]
		Write $(-)^*=\Hom_{A^{\mathrm{op}}}(-,A)$.  Since
		\[
		\Ext^r_{A^{\mathrm{op}}}(X,A)=0
		\qquad (0\le r<q),
		\]
		applying $(-)^*$ to the resolution and taking a cokernel gives an
		exact sequence
		\[
		0\longrightarrow P_0^*\longrightarrow P_1^*\longrightarrow\cdots\longrightarrow P_q^*\longrightarrow M\longrightarrow0,
		\]
		where
		\[
		M=\operatorname{coker}(P_{q-1}^*\longrightarrow P_q^*).
		\]
		This is a projective resolution of the left $A$-module $M$ of length
		$q$, with $P_q^*$ in cohomological degree $0$ and $P_0^*$ in degree
		$-q$.  After applying $\Hom_A(-,A)$, the latter term occurs in degree
		$q$.  Using $P_r^{**}\cong P_r$, we obtain
		\[
		\Ext_A^q(M,A)
		\cong
		\operatorname{coker}(P_1\longrightarrow P_0)
		\cong X.
		\]
	\end{proof}
	
	For $q\ge1$, put
	\[
	\mathcal G_q(A)
	=
	\{X\in\modcat\text{-}A\mid \grade_{A^{\mathrm{op}}} X\ge q\}.
	\]
	
	We now show that these modules form a Serre subcategory.
	
	\begin{proposition}\label{prop:grade-serre}
		If $A$ satisfies the Auslander condition, then $\mathcal G_q(A)$ is a Serre subcategory of $\modcat\text{-}A$ for every $q\ge1$.
	\end{proposition}
	
	\begin{proof}
		We verify closure under extensions, submodules and quotients.  Closure
		under extensions follows from the long exact Ext sequence.  Let
		$U\subseteq X$ with $X\in\mathcal G_q(A)$.  By
		Lemma~\ref{lem:realization}, we may write $X\cong\Ext_A^q(M,A)$.
		The Auslander condition then gives $\grade_{A^{\mathrm{op}}}U\ge q$,
		so $U\in\mathcal G_q(A)$.  Finally, consider an exact sequence
		\[
		0\longrightarrow U\longrightarrow X\longrightarrow Q\longrightarrow0
		\]
		with $U,X\in\mathcal G_q(A)$.  The long exact Ext sequence gives
		\[
		\Ext^r_{A^{\mathrm{op}}}(Q,A)=0
		\qquad (0\le r<q),
		\]
		so $Q\in\mathcal G_q(A)$.
	\end{proof}
	
	\section{Local cohomology categories are abelian except in top degree}\label{sec:finitebox}
	
	In this section, we  apply the homological results of Section~\ref{sec:homological}
	to determine the image categories $\mathcal H_i(\mathbf t)$.
	In the proofs, we use the exact equivalence $\theta_{\mathbf t}$ to pass
	from $\mathcal H_i(\mathbf t)$ to $\mathcal T_i(\mathbf t)$.
	
	The key input is a formula identifying the grade of a vertex simple
	with its support rank.  The poset $P_{\mathbf t}^{\mathrm{op}}$ is a finite distributive lattice.
	Iyama and Marczinzik work with right modules and use the incidence-algebra
	convention opposite to ours.  In their convention, the incidence algebra
	of $P_{\mathbf t}^{\mathrm{op}}$ is $\Lambda_{\mathbf t}$.  Thus
	\cite[Theorem~3.2]{IM} shows that $\Lambda_{\mathbf t}$ is Auslander regular.
	Applying \cite[Corollary~3.3]{IM} to $P_{\mathbf t}^{\mathrm{op}}$ gives,
	for the right simple $R_{\mathbf a}$,
	\[
	\grade_{\Lambda_{\mathbf t}^{\mathrm{op}}}R_{\mathbf a}
	=
	\#\{\text{elements of $P_{\mathbf t}$ covered by $\mathbf a$}\}
	=
	\ell(\mathbf a).
	\]
	Indeed, the elements covering $\mathbf a$ in $P_{\mathbf t}^{\mathrm{op}}$
	are precisely $\mathbf a-\mathbf e_j$ with $a_j>0$.  Hence their number
	is $\ell(\mathbf a)$.
	
	\subsection{The layers below the top}
	
	\begin{theorem}\label{thm:finite-intrinsic}
		Let $0\le i<n$ and put $q=n-i$.  Then
		\[
		\mathcal H_i(\mathbf t)
		=\{L\in\mathcal A^-_{\mathbf t}\mid
		L_{\mathbf a}=0\text{ for all }\mathbf a\in P_{\mathbf t}
		\text{ with }\ell(\mathbf a)<q\}.
		\]
		Consequently $\mathcal H_i(\mathbf t)$ is a Serre subcategory of
		$\mathcal A^-_{\mathbf t}$, and
		$\theta_{\mathbf t}$ restricts to an exact equivalence
		\[
		\theta_{\mathbf t}|_{\mathcal H_i(\mathbf t)}:
		\mathcal H_i(\mathbf t)\xrightarrow{\sim}\Rep(U_q(\mathbf t)).
		\]
	\end{theorem}
	
	\begin{proof}
		Under $\theta_{\mathbf t}$, it suffices to prove
		\[
		\mathcal T_i(\mathbf t)
		=\{V\in\Rep(P_{\mathbf t}):V_{\mathbf a}=0
		\text{ whenever }\ell(\mathbf a)<q\}
		=\Rep(U_q(\mathbf t)),
		\]
		where the last category is identified with its image under extension by zero.
		Since $q\ge1$ and $\Lambda_{\mathbf t}$ satisfies the Auslander
		condition, Proposition~\ref{prop:naka-ext} and
		Lemma~\ref{lem:realization} give
		\[
		\mathcal T_i(\mathbf t)
		=
		D\bigl(\mathcal G_q(\Lambda_{\mathbf t})\bigr).
		\]
		By Proposition~\ref{prop:grade-serre}, $\mathcal G_q(\Lambda_{\mathbf t})$
		is a Serre subcategory.  The duality $D$ is exact, so
		$\mathcal T_i(\mathbf t)$ is also a Serre subcategory.
		
		It remains to determine its objects.  Every finite-dimensional
		$\Lambda_{\mathbf t}$-module has finite length.  Hence a Serre
		subcategory is determined by the simple modules it contains.  The grade
		formula above shows that $R_{\mathbf a}\in\mathcal G_q(\Lambda_{\mathbf t})$
		if and only if
		\[
		\grade_{\Lambda_{\mathbf t}^{\mathrm{op}}}R_{\mathbf a}=\ell(\mathbf a)\ge q.
		\]
		Since $D$ preserves vertex support, the corresponding representations
		have no composition factors at vertices with $\ell(\mathbf a)<q$.
		The multiplicity of the vertex simple $S_{\mathbf a}$ in a composition
		series of $V$ is $\dim_\kk V_{\mathbf a}$, by additivity of dimension
		vectors.  Thus this condition is equivalent to
		\[
		V_{\mathbf a}=0\qquad\text{for }\ell(\mathbf a)<q.
		\]
		Since $\ell(\mathbf a)\le\ell(\mathbf b)$ whenever
		$\mathbf a\le\mathbf b$, the subposet $U_q(\mathbf t)$ is an upper order
		ideal of $P_{\mathbf t}$.  Hence extension by zero identifies
		$\Rep(U_q(\mathbf t))$ with the full subcategory of
		$\Rep(P_{\mathbf t})$ consisting of representations $V$ such that
		$\operatorname{supp}V\subseteq U_q(\mathbf t)$.  This proves the assertion.
	\end{proof}
	
	The same support description shows how the image categories change
	with the cohomological degree.
	
	\begin{corollary}\label{cor:finite-filtration}
		The Serre subcategories $\mathcal H_i(\mathbf t)$ of
		$\mathcal A^-_{\mathbf t}$ form an increasing chain
		\[
		\mathcal H_0(\mathbf t)
		\subseteq\mathcal H_1(\mathbf t)
		\subseteq\cdots\subseteq\mathcal H_{n-1}(\mathbf t).
		\]
		For $1\le i<n$, the new simple objects appearing in the passage from
		$\mathcal H_{i-1}(\mathbf t)$ to $\mathcal H_i(\mathbf t)$ are
		$\sigma_{\mathbf t}S_{\mathbf a}$ with
		\[
		\ell(\mathbf a)=n-i.
		\]
		If $\mathbf t\ge\mathbf1$, all inclusions are strict.
	\end{corollary}
	
	\begin{remark}\label{rem:projectives-under-inclusions}
		For $0\le i<n$, the inclusion
		$\mathcal H_i(\mathbf t)\hookrightarrow\mathcal A^-_{\mathbf t}$
		preserves projective objects, and
		\[
		\proj\mathcal H_i(\mathbf t)
		=\proj\mathcal A^-_{\mathbf t}\cap\mathcal H_i(\mathbf t).
		\]
		For every $N\in\mathcal H_i(\mathbf t)$, one has
		\[
		\pd_{\mathcal H_i(\mathbf t)}N
		=\pd_{\mathcal A^-_{\mathbf t}}N.
		\]
		Consequently, for $0\le i\le n-2$, the inclusion
		$\mathcal H_i(\mathbf t)\hookrightarrow\mathcal H_{i+1}(\mathbf t)$
		also preserves projective objects, and
		\[
		\proj\mathcal H_i(\mathbf t)
		=\proj\mathcal H_{i+1}(\mathbf t)\cap\mathcal H_i(\mathbf t),
		\qquad
		\pd_{\mathcal H_i(\mathbf t)}N
		=\pd_{\mathcal H_{i+1}(\mathbf t)}N
		\quad(N\in\mathcal H_i(\mathbf t)).
		\]
		Indeed, under the finite-box equivalences, these inclusions are extension
		by zero along the upper order ideals
		$U_{n-i}(\mathbf t)\subseteq P_{\mathbf t}$ and, when $i\le n-2$,
		$U_{n-i}(\mathbf t)\subseteq U_{n-i-1}(\mathbf t)$.
		The principal upper interval at a vertex of the smaller poset is unchanged
		in the larger poset, so the corresponding indecomposable projective
		remains projective.  This proves the assertions about projective objects.
		We can therefore compute Ext between objects of the smaller category
		using the same projective resolution in the smaller and larger categories.
		Fullness identifies the resulting Hom complexes.  These Ext isomorphisms,
		together with preservation of projective resolutions, give the
		projective-dimension equalities.
	\end{remark}
	
	For $0\le i<n$, Theorem~\ref{thm:finite-intrinsic} yields
	\[
	\mathcal H_i(\mathbf t)\ne0
	\quad\Longleftrightarrow\quad
	U_{n-i}(\mathbf t)\ne\varnothing
	\quad\Longleftrightarrow\quad
	\ell(\mathbf t)\ge n-i.
	\]
	We next compute the global dimensions of the nonzero categories
	$\mathcal H_i(\mathbf t)$.
	
	\begin{proposition}
		\label{prop:image-global-dimension}
		Let $0\le i<n$ and suppose $\mathcal H_i(\mathbf t)\ne0$.  Set
		$r_{\mathbf t}=\#\{j:1\le j\le n,\ t_j\ge2\}$.  Then
		\[
		\operatorname{gl.dim}\mathcal H_i(\mathbf t)
		=\min\{\ell(\mathbf t),\ell(\mathbf t)+i+r_{\mathbf t}-n\}.
		\]
		In particular, if $\mathbf t\ge\mathbf1$, then
		\[
		\operatorname{gl.dim}\mathcal H_i(\mathbf t)
		=\min\{n,i+r_{\mathbf t}\}.
		\]
		For $\mathbf t=\mathbf1$, this global dimension is $i$.
	\end{proposition}
	
	\begin{proof}
		Put $q=n-i$ and fix $\mathbf a\in U_q(\mathbf t)$.  For each
		$1\le j\le n$, let $P_b^{(j)}$ and $T_b^{(j)}$ denote the principal
		projective and vertex simple at $b$ in $\Rep([0,t_j])$, respectively.
		If $b<t_j$, there is a minimal projective resolution
		\[
		0\longrightarrow P_{b+1}^{(j)}\longrightarrow P_b^{(j)}
		\longrightarrow T_b^{(j)}\longrightarrow0,
		\]
		while $T_{t_j}^{(j)}$ is projective.  Put
		$B(\mathbf a)=\{j:1\le j\le n,\ a_j<t_j\}$.
		Tensoring these resolutions over $\kk$ gives a resolution of the vertex
		simple $T_{\mathbf a}$ whose $h$th term is
		\[
		\bigoplus_{\substack{F\subseteq B(\mathbf a)\\|F|=h}}
		P_{\mathbf a+\sum_{j\in F}\mathbf e_j},
		\]
		where $P_{\mathbf b}$ denotes the principal projective at $\mathbf b$
		in $\Rep(P_{\mathbf t})$.  Since $U_q(\mathbf t)$ is an upper order ideal,
		each displayed summand is supported on $U_q(\mathbf t)$ and is
		projective in $\Rep(U_q(\mathbf t))$.  Restriction preserves exactness.
		The differentials have only radical components between distinct
		indecomposable projectives, so the resolution is minimal.  Thus
		\[
		\pd_{\Rep(U_q(\mathbf t))}T_{\mathbf a}=|B(\mathbf a)|.
		\]
		A finite-dimensional algebra has global dimension equal to the maximum
		projective dimension of its simple modules.  The equivalence in
		Theorem~\ref{thm:finite-intrinsic} therefore gives
		\[
		\operatorname{gl.dim}\mathcal H_i(\mathbf t)
		=\max\bigl\{\#\{j:1\le j\le n,\ a_j<t_j\}
		\mid\mathbf a\in U_q(\mathbf t)\bigr\}.
		\]
		
		Put $m=\ell(\mathbf t)$.  Coordinates with $t_j=0$ do not contribute to
		this count, since then $a_j=0$.  A vertex in $U_q(\mathbf t)$ has at
		least $q$ positive coordinates, of which at most $r_{\mathbf t}$ can
		be strictly smaller than their upper endpoints.  Hence at least
		$\max\{0,q-r_{\mathbf t}\}$ of these positive coordinates satisfy
		$a_j=t_j$, so the count is at most $m-\max\{0,q-r_{\mathbf t}\}$.
		This bound is attained by taking $a_j=1$ in
		$\min\{q,r_{\mathbf t}\}$ coordinates with $t_j\ge2$ and in
		$q-\min\{q,r_{\mathbf t}\}$ coordinates with $t_j=1$, and taking
		$a_j=0$ elsewhere.  This choice is possible because $q\le m$, so
		$q-\min\{q,r_{\mathbf t}\}\le m-r_{\mathbf t}$.  Therefore
		\[
		\operatorname{gl.dim}\mathcal H_i(\mathbf t)
		=m-\max\{0,q-r_{\mathbf t}\}
		=\min\{m,m+i+r_{\mathbf t}-n\}.
		\]
		If $\mathbf t\ge\mathbf1$, then $m=n$; if $\mathbf t=\mathbf1$, then
		also $r_{\mathbf t}=0$.  The special cases follow.
	\end{proof}
	
	The inclusions in Corollary~\ref{cor:finite-filtration} can also be
	obtained by taking syzygies.  Recall that the minimal homogeneous
	generators of a positively $\mathbf t$-determined module have
	multidegrees $\le\mathbf t$.  Conversely, a torsion-free
	$\mathbb N^n$-graded module generated in these degrees is positively
	$\mathbf t$-determined.  We use this observation in the next proof.
	
	\begin{proposition}\label{prop:syzygy-filtration}
		Let $M\in\mathcal A^+_{\mathbf t}$, let $P\twoheadrightarrow M$ be a minimal graded free cover, and put
		$$
		\Omega M=\ker(P\to M).
		$$
		Then $\Omega M\in\mathcal A^+_{\mathbf t}$. Moreover, for
		$0\le i\le n-2$,
		$$
		H^i_{\mathfrak m}(M(-\mathbf1))
		\cong
		H^{i+1}_{\mathfrak m}(\Omega M(-\mathbf1)).
		$$
	\end{proposition}
	
	\begin{proof}
		Since $M$ is positively $\mathbf t$-determined, all its minimal homogeneous generators have multidegrees $\le \mathbf t$. Hence its minimal graded free cover has the form
		$$
		P=
		\bigoplus_{\mathbf b\in P_{\mathbf t}}
		S(-\mathbf b)^{\beta_{0,\mathbf b}(M)}.
		$$
		
		Thus $P$ is torsion-free and generated in multidegrees $\le\mathbf t$,
		so $P\in\mathcal A^+_{\mathbf t}$.  This category is closed under
		kernels.  Since $M$ also belongs to it, we obtain
		$$
		\Omega M=\ker(P\to M)\in\mathcal A^+_{\mathbf t}.
		$$
		
		Consider the short exact sequence
		$$
		0\longrightarrow\Omega M\longrightarrow P\longrightarrow M\longrightarrow0
		$$
		and shift it by $-\mathbf1$.  Since $P$ is free,
		$H^r_{\mathfrak m}(P(-\mathbf1))=0$ for $0\le r<n$.  The long exact
		sequence in local cohomology gives the required isomorphism.
	\end{proof}

	\begin{remark}\label{rem:dimension-shift-stops}
		Proposition~\ref{prop:syzygy-filtration} gives
		$\mathcal H_i(\mathbf t)\subseteq\mathcal H_{i+1}(\mathbf t)$ for
		$0\le i\le n-2$.  The same argument does not give an inclusion in
		top degree.  Indeed, when $i=n-1$, the long exact sequence contains
		\[
		0\longrightarrow H^{n-1}_{\mathfrak m}(M(-\mathbf1))
		\longrightarrow H^n_{\mathfrak m}(\Omega M(-\mathbf1))
		\longrightarrow H^n_{\mathfrak m}(P(-\mathbf1)),
		\]
		and the last term is generally nonzero.  Thus dimension shifting does not
		give the required isomorphism in top degree.
	\end{remark}

	Within the non-top degrees, we can say more about the increasing chain:
	it yields a functorial torsion decomposition in each image category.  For $0\le s<n$, let $\mathcal B_s(\mathbf t)$
	be the full subcategory of $\mathcal H_s(\mathbf t)$ consisting of those
	$N$ for which $\theta_{\mathbf t}N$ is supported on
	\[
	L_{n-s}(\mathbf t)
	=\{\mathbf a\in P_{\mathbf t}:\ell(\mathbf a)=n-s\}.
	\]
	In particular, $\mathcal B_0(\mathbf t)=\mathcal H_0(\mathbf t)$.
	
	We use the standard concept that a pair $(\mathcal T,\mathcal F)$ of full
	subcategories is a \emph{torsion pair} if
	$\Hom(\mathcal T,\mathcal F)=0$ and every object $E$ occurs in a short exact
	sequence $0\to T\to E\to F\to0$ with $T\in\mathcal T$ and
	$F\in\mathcal F$.  It is \emph{hereditary} if $\mathcal T$ is closed under
	subobjects, equivalently if $\mathcal T$ is a Serre subcategory.
	
	\begin{proposition}\label{prop:adjacent-torsion-pair}
		For $0\le i\le n-2$, the pair
		\[
		\bigl(\mathcal H_i(\mathbf t),\mathcal B_{i+1}(\mathbf t)\bigr)
		\]
		is a hereditary torsion pair in $\mathcal H_{i+1}(\mathbf t)$.
		Every $N\in\mathcal H_{i+1}(\mathbf t)$ has a canonical, functorial
		short exact sequence
		\[
		0\longrightarrow T_iN\longrightarrow N\longrightarrow F_{i+1}N
		\longrightarrow0,
		\qquad
		T_iN\in\mathcal H_i(\mathbf t),\quad
		F_{i+1}N\in\mathcal B_{i+1}(\mathbf t).
		\]
		Moreover, $T_iN$ is the largest submodule of $N$ belonging to
		$\mathcal H_i(\mathbf t)$.
	\end{proposition}
	
	\begin{proof}
		Put $q=n-i$ and $V=\theta_{\mathbf t}N$.  By
		Theorem~\ref{thm:finite-intrinsic}, $V$ is supported on
		$U_{q-1}(\mathbf t)$.  Since $U_q(\mathbf t)$ is an upper order ideal,
		\[
		(W_iV)_{\mathbf a}=
		\begin{cases}
			V_{\mathbf a},&\ell(\mathbf a)\ge q,\\
			0,&\ell(\mathbf a)<q
		\end{cases}
		\]
		defines a subrepresentation of $V$.  It is the largest subrepresentation
		supported on $U_q(\mathbf t)$, and $V/W_iV$ is supported on
		$L_{q-1}(\mathbf t)$.  Applying the exact quasi-inverse
		$\sigma_{\mathbf t}$ to $0\to W_iV\to V\to V/W_iV\to0$ gives the
		asserted sequence and maximality.  Every morphism of representations
		preserves these subrepresentations, so the sequence is functorial.
		The supports of an object supported on $U_q(\mathbf t)$ and an object
		supported on $L_{q-1}(\mathbf t)$ are disjoint; every morphism from the
		former to the latter is therefore zero.  This proves the torsion-pair
		assertion.  It is hereditary because $\mathcal H_i(\mathbf t)$ is a
		Serre subcategory of $\mathcal H_{i+1}(\mathbf t)$.
	\end{proof}
	
	Iterating these torsion decompositions gives the following filtration.
	
	\begin{corollary}\label{cor:canonical-layer-filtration}
		For $0\le r<n$, every $N\in\mathcal H_r(\mathbf t)$ has a canonical,
		functorial filtration by submodules
		\[
		0=T_{-1}N\subseteq T_0N\subseteq\cdots\subseteq T_rN=N,
		\]
		where $T_sN$ is the largest submodule of $N$ belonging to
		$\mathcal H_s(\mathbf t)$ and
		\[
		T_sN/T_{s-1}N\in\mathcal B_s(\mathbf t)
		\qquad(0\le s\le r).
		\]
	\end{corollary}
	
	\begin{proof}
		For $V=\theta_{\mathbf t}N$, take its subrepresentation equal to $V$
		on $U_{n-s}(\mathbf t)$ and zero elsewhere, and transport it through
		$\sigma_{\mathbf t}$ to define $T_sN$.  These upper order ideals are
		nested, giving the displayed filtration.  Each term is maximal with the
		required support, and each quotient is supported on
		$L_{n-s}(\mathbf t)$.  The construction is preserved by every morphism,
		which proves functoriality.  Equivalently, the filtration is obtained by
		successively applying Proposition~\ref{prop:adjacent-torsion-pair}.
	\end{proof}
	
	Thus the chain of image categories gives a canonical filtration whose
	successive factors are supported on individual support-rank layers.
	
	\subsection{The top-degree category and second cosyzygies}
	
	To describe $\mathcal H_n(\mathbf t)$, we first work in the finite-box
	model.  Proposition~\ref{prop:naka-ext} gives
	\[
	\mathcal T_n(\mathbf t)
	=
	\EssIm\left(
	D\circ\Hom_{\Lambda_{\mathbf t}}(-,\Lambda_{\mathbf t})
	:\Lambda_{\mathbf t}\text{-mod}\longrightarrow\Rep(P_{\mathbf t})
	\right).
	\]
	Lemma~\ref{lem:realization} applies only to positive Ext degrees.
	To treat degree zero, we first describe the essential image of
	$\Hom_A(-,A)$ in terms of second syzygies.
	
	For a finite-dimensional $\kk$-algebra $A$, let
	$\operatorname{Syz}_2(\modcat\text{-}A)$ denote the full subcategory of
	$\modcat\text{-}A$ consisting of all modules $X$ for which there exists an
	exact sequence
	$$
	0\longrightarrow X\longrightarrow P_1\longrightarrow P_0
	\longrightarrow C\longrightarrow0
	$$
	with $P_0$ and $P_1$ projective.  The resolution need not be minimal.  In particular, every projective
	module belongs to
	$\operatorname{Syz}_2(\modcat\text{-}A)$.
	
	\begin{lemma}\label{lem:dual-syzygy}
		For every finite-dimensional $\kk$-algebra $A$,
		$$
		\EssIm\!\left(
		\Hom_A(-,A):(A\text{-}\modcat)^{\mathrm{op}}\longrightarrow\modcat\text{-}A
		\right)
		=
		\operatorname{Syz}_2(\modcat\text{-}A).
		$$
	\end{lemma}
	
	\begin{proof}
		Let $M\in A\text{-}\modcat$, put $M^*=\Hom_A(M,A)$, and choose a projective
		presentation
		$$
		P_1\longrightarrow P_0\longrightarrow M\longrightarrow0.
		$$
		Applying $(-)^*=\Hom_A(-,A)$ gives an exact sequence
		$$
		0\longrightarrow M^*\longrightarrow P_0^*\longrightarrow P_1^*.
		$$
		The modules $P_0^*$ and $P_1^*$ are projective right $A$-modules.  Hence
		$M^*\in\operatorname{Syz}_2(\modcat\text{-}A)$.
		
		Conversely, suppose $X\in\operatorname{Syz}_2(\modcat\text{-}A)$,
		and choose an exact sequence
		$$
		0\longrightarrow X\longrightarrow Q_1\xrightarrow{d}Q_0
		\longrightarrow C\longrightarrow0
		$$
		with $Q_0$ and $Q_1$ projective right $A$-modules.  Put
		$$
		M=\operatorname{coker}(d^*:Q_0^*\longrightarrow Q_1^*).
		$$
		Then $Q_0^*\to Q_1^*\to M\to0$ is a projective presentation of the left
		$A$-module $M$.  Dualizing it and using the natural isomorphisms
		$Q_i^{**}\cong Q_i$ gives
		$$
		M^*\cong\ker(d)=X.
		$$
		Thus $X$ belongs to the essential image of $\Hom_A(-,A)$.
	\end{proof}
	
	For $\mathbf a\in P_{\mathbf t}$ let $I^{\mathbf t}_{\mathbf a}$ denote the principal lower-interval representation
	\[
	I^{\mathbf t}_{\mathbf a}(\mathbf b)=
	\begin{cases}
		\kk,&\mathbf b\le\mathbf a,\\
		0,&\text{otherwise},
	\end{cases}
	\]
	with identity maps between nonzero components.  These are precisely the indecomposable injective objects of $\Rep(P_{\mathbf t})$.
	
	An object $V$ of an abelian category is called a \emph{second cosyzygy}
	if there exists an exact sequence
	$$
	0\longrightarrow W\longrightarrow I^0\longrightarrow I^1
	\longrightarrow V\longrightarrow0
	$$
	with $I^0$ and $I^1$ injective.
	
	\begin{theorem}\label{thm:top-finite}
		The category $\mathcal H_n(\mathbf t)$ consists precisely of the second
		cosyzygies in $\mathcal A^-_{\mathbf t}$.  In particular, it contains
		every injective object of $\mathcal A^-_{\mathbf t}$ and is
		extension-closed in $\mathcal A^-_{\mathbf t}$.
	\end{theorem}
	
	\begin{proof}
		By Lemma~\ref{lem:dual-syzygy},
		\[
		\mathcal T_n(\mathbf t)
		=D\bigl(\operatorname{Syz}_2(\modcat\text{-}\Lambda_{\mathbf t})\bigr).
		\]
		Since $D$ is an exact duality sending projectives to injectives,
		$\mathcal T_n(\mathbf t)$ is precisely the category of second cosyzygies
		in $\Rep(P_{\mathbf t})$.  In particular, it contains every injective
		representation.  Since $\Lambda_{\mathbf t}$ is Auslander regular,
		the second syzygies are extension-closed by Auslander--Reiten
		\cite[Corollary~3.1]{ARsyzygy}. Applying $D$ shows that
		$\mathcal T_n(\mathbf t)$ is extension-closed.
		The exact equivalence $\theta_{\mathbf t}$ preserves injectives and
		second cosyzygies, so all assertions follow for $\mathcal H_n(\mathbf t)$.
	\end{proof}
	
	This description also separates the top-degree image from all non-top
	images.
	
	\begin{proposition}\label{prop:top-disjoint}
		For every $0\le i<n$, one has
		\[
		\mathcal H_i(\mathbf t)\cap\mathcal H_n(\mathbf t)=0.
		\]
	\end{proposition}
	
	\begin{proof}
		Applying $\theta_{\mathbf t}$ and using
		Theorems~\ref{thm:finite-intrinsic} and~\ref{thm:top-finite},
		it is enough to show
		\[
		\operatorname{Syz}_2(\modcat\text{-}\Lambda_{\mathbf t})
		\cap
		\mathcal G_1(\Lambda_{\mathbf t})=0.
		\]
		Let $X$ be a nonzero second syzygy.  Then $X$ embeds in a finitely
		generated projective right module $P$, which is a direct summand of
		$\Lambda_{\mathbf t}^m$ for some $m$.  The resulting map
		$X\to\Lambda_{\mathbf t}^m$ has a nonzero coordinate map to
		$\Lambda_{\mathbf t}$.  Therefore
		\[
		\Hom_{\Lambda_{\mathbf t}^{\mathrm{op}}}(X,\Lambda_{\mathbf t})\ne0,
		\]
		so $\grade_{\Lambda_{\mathbf t}^{\mathrm{op}}}X=0$.  Hence $X\notin\mathcal G_1(\Lambda_{\mathbf t})$.
	\end{proof}
	
	Although the top-degree image is extension-closed, it is not abelian
	when the box has a nonzero vertex.  To prove this, we use the following
	criterion.  Recall that an additive category $\mathcal C$ is called
	\emph{balanced} if every morphism in $\mathcal C$ that is both a monomorphism
	and an epimorphism is an isomorphism.  Every abelian category is balanced:
	if $f:X\to Y$ is both monic and epic, then $\ker f=0$ and
	$\operatorname{coker}f=0$, and the canonical exact sequence
	\[
	0\longrightarrow\ker f\longrightarrow X\xrightarrow{f}Y
	\longrightarrow\operatorname{coker}f\longrightarrow0
	\]
	shows that $f$ is an isomorphism.
	
	\begin{proposition}\label{prop:top-nonabelian}
		Assume that $P_{\mathbf t}$ has a nonzero vertex, equivalently
		$\mathbf t\ne\mathbf0$.  Then $\mathcal H_n(\mathbf t)$ is not an
		abelian category.
	\end{proposition}
	
	\begin{proof}
		It suffices to show that the equivalent category
		$\mathcal T_n(\mathbf t)$ is not balanced.
		Choose $\mathbf a\ne\mathbf0$ and let
		\[
		\pi:I^{\mathbf t}_{\mathbf a}\longrightarrow
		I^{\mathbf t}_{\mathbf0}=S_{\mathbf0}
		\]
		be the canonical epimorphism which is the identity at $\mathbf0$ and zero
		elsewhere.  Both terms lie in $\mathcal T_n(\mathbf t)$.  Its kernel $K$ in
		$\Rep(P_{\mathbf t})$ is nonzero and satisfies $K_{\mathbf0}=0$; hence
		$K\in\mathcal T_{n-1}(\mathbf t)$ by
		Theorem~\ref{thm:finite-intrinsic}, and therefore
		$K\notin\mathcal T_n(\mathbf t)$ by Proposition~\ref{prop:top-disjoint}.
		Thus the ambient kernel does not lie in $\mathcal T_n(\mathbf t)$.
		We claim, more precisely, that $\pi$ is both monic and epic in
		$\mathcal T_n(\mathbf t)$, although it is not an isomorphism.
		
		To prove the claim, let $V\in\mathcal T_n(\mathbf t)$. By
		Theorem~\ref{thm:top-finite}, $V$ is the cokernel of a morphism
		between finite direct sums of principal injectives. Hence each structure map
		$V_{\mathbf0}\to V_{\mathbf b}$ is surjective:
		this is true for every principal injective and is preserved by finite direct
		sums and cokernels.  Since $K_{\mathbf0}=0$, naturality gives
		$\Hom(V,K)=0$.  Thus $\pi$ is a monomorphism in
		$\mathcal T_n(\mathbf t)$; it is also an epimorphism there, since it is one in
		$\Rep(P_{\mathbf t})$, but it is not an isomorphism.  An abelian category is
		balanced, so $\mathcal T_n(\mathbf t)$ is not abelian.  The assertion for
		$\mathcal H_n(\mathbf t)$ follows from $\theta_{\mathbf t}$.
	\end{proof}
	
	Combining Proposition~\ref{prop:top-disjoint} with
	Theorem~\ref{thm:finite-intrinsic} for $i=n-1$ gives
	\[
	0\ne L\in\mathcal H_n(\mathbf t)
	\quad\Longrightarrow\quad
	L_{\mathbf0}\ne0.
	\]
	Put $\mathcal E_{\mathbf0}=\sigma_{\mathbf t}S_{\mathbf0}$.
	Up to isomorphism, it is the only simple object of $\mathcal A^-_{\mathbf t}$
	belonging to $\mathcal H_n(\mathbf t)$.  Indeed,
	$S_{\mathbf0}=I^{\mathbf t}_{\mathbf0}$ is injective, so
	$\mathcal E_{\mathbf0}\in\mathcal H_n(\mathbf t)$, while
	$\sigma_{\mathbf t}S_{\mathbf a}\in\mathcal H_{n-1}(\mathbf t)$
	for every $\mathbf a\ne\mathbf0$.  The assertion follows from
	Theorem~\ref{thm:top-finite} and Proposition~\ref{prop:top-disjoint}.
	
	The simple object $\mathcal E_{\mathbf0}$ supplies the remaining factor
	in a torsion decomposition of the ambient category.  To state this
	precisely, for an object $X$ of an additive category, write $\operatorname{add}X$ for
	the full subcategory of direct summands of finite direct sums of copies of
	$X$.  For full additive subcategories $\mathcal X,\mathcal Y$ of an abelian
	category, write $\mathcal X\bullet\mathcal Y$ for the additive closure of
	objects $E$ occurring in a short exact sequence
	\[
	0\longrightarrow X\longrightarrow E\longrightarrow Y\longrightarrow0,
	\qquad X\in\mathcal X,\quad Y\in\mathcal Y.
	\]
	\begin{proposition}\label{prop:top-penultimate}
		Inside $\mathcal A^-_{\mathbf t}$ one has
		\[
		\mathcal A^-_{\mathbf t}
		=
		\mathcal H_{n-1}(\mathbf t)\bullet\operatorname{add}\mathcal E_{\mathbf0}
		=
		\mathcal H_{n-1}(\mathbf t)\bullet\mathcal H_n(\mathbf t).
		\]
		Moreover
		\[
		\bigl(\mathcal H_{n-1}(\mathbf t),\operatorname{add}\mathcal E_{\mathbf0}\bigr)
		\]
		is a hereditary torsion pair in $\mathcal A^-_{\mathbf t}$.  In general
		$\bigl(\mathcal H_{n-1}(\mathbf t),\mathcal H_n(\mathbf t)\bigr)$ is
		not a torsion pair.
	\end{proposition}
	
	\begin{proof}
		We prove the corresponding statements in $\Rep(P_{\mathbf t})$ and
		then apply the exact quasi-inverse $\sigma_{\mathbf t}$.
		By Theorem~\ref{thm:finite-intrinsic},
		\[
		\mathcal T_{n-1}(\mathbf t)
		=
		\{V\in\Rep(P_{\mathbf t})\mid V_{\mathbf0}=0\}.
		\]
		For $V\in\Rep(P_{\mathbf t})$, define a subrepresentation $V^{>0}$ by
		\[
		(V^{>0})_{\mathbf0}=0,
		\qquad
		(V^{>0})_{\mathbf a}=V_{\mathbf a}
		\quad(\mathbf a\ne\mathbf0).
		\]
		This is a subrepresentation because $\mathbf0$ is the minimal element of
		$P_{\mathbf t}$.  Hence
		$V^{>0}\in\mathcal T_{n-1}(\mathbf t)$ and
		\[
		V/V^{>0}\cong
		S_{\mathbf0}^{\,\dim_\kk V_{\mathbf0}}.
		\]
		The sequence $0\to V^{>0}\to V\to V/V^{>0}\to0$ gives the first
		extension equality.  Since $S_{\mathbf0}\in\mathcal T_n(\mathbf t)$,
		the second equality follows as well.
		
		If $X\in\mathcal T_{n-1}(\mathbf t)$, then $X_{\mathbf0}=0$, and hence
		$\Hom(X,S_{\mathbf0})=0$.  The short exact sequence just constructed therefore
		shows that
		$\bigl(\mathcal T_{n-1}(\mathbf t),\operatorname{add}S_{\mathbf0}\bigr)$
		is a torsion pair.  It is hereditary because
		$\mathcal T_{n-1}(\mathbf t)$ is a Serre subcategory.
		
		Finally, if $\mathbf a\ne\mathbf0$, then
		$S_{\mathbf a}\in\mathcal T_{n-1}(\mathbf t)$ and the principal
		lower-interval representation
		$I^{\mathbf t}_{\mathbf a}$ belongs to $\mathcal T_n(\mathbf t)$,
		while
		\[
		\Hom(S_{\mathbf a},I^{\mathbf t}_{\mathbf a})\ne0.
		\]
		Thus the required Hom-orthogonality for
		$\bigl(\mathcal T_{n-1}(\mathbf t),\mathcal T_n(\mathbf t)\bigr)$
		fails whenever the box has a nonzero vertex.
		Applying $\sigma_{\mathbf t}$ gives all assertions in
		$\mathcal A^-_{\mathbf t}$.
	\end{proof}

	\subsection{An alternative description of $\mathcal H_i(\mathbf t)$ for $i\in\{n-1,n\}$}
	
	The preceding descriptions use the finite-box categories.  We now
	characterize the top two images in terms of torsion and reflexivity of
	their graded Matlis duals.
	
	For a finitely generated $S$-module $L$, we use
	$\operatorname{rank}_S L=\dim_K(L\otimes_S K)$, where
	$K=\operatorname{Frac}(S)$; see \cite[Definition~1.4.2]{BH}.
	
	For a finitely generated multigraded $S$-module $L$, write
	$L^*=\underline{\Hom}_S(L,S)$ for its graded $S$-dual and
	$\eta_L:L\to L^{**}$ for the evaluation map.  This duality is distinct
	from the graded $\kk$-dual $\Dgr$.  We call $L$ reflexive if $\eta_L$
	is an isomorphism.
	
	\begin{theorem}\label{thm:reflexive-top-complement}
		Assume $n\ge1$.  Let $N\in\mathcal A^-_{\mathbf t}$ and put $L=\Dgr N$.
		\begin{enumerate}[label=\textup{(\roman*)},leftmargin=*]
			\item One has
			\[
			N\in\mathcal H_{n-1}(\mathbf t)
			\ \Longleftrightarrow\ L\text{ is $S$-torsion}.
			\]
			\item One has
			\[
			N\in\mathcal H_n(\mathbf t)
			\ \Longleftrightarrow\ L\text{ is reflexive}.
			\]
			If $L$ is reflexive, then $L^*\in\mathcal A^+_{\mathbf t}$ and
			$\Phi_n^{\mathbf t}(L^*)\cong N$.
		\end{enumerate}
	\end{theorem}
	
	\begin{proof}[Proof of \textup{(i)}]
		Negative determination of $N$ implies that $L$ is supported in
		$\{\mathbf a\in\mathbb Z^n:\mathbf a\ge-\mathbf t\}$ and that multiplication by $x_j$
		on $L_{\mathbf a}$ is an isomorphism whenever $a_j\ge0$.
		Thus $L(-\mathbf t)\in\mathcal A^+_{\mathbf t}$.
		After inverting all variables, the degree-zero component of $L$ is the
		direct limit of its components in nonnegative degrees.  All maps in
		this system are isomorphisms, so this component is $L_{\mathbf0}$.
		The resulting multigraded module over
		$\kk[x_1^{\pm1},\ldots,x_n^{\pm1}]$ is free of rank
		$\dim_\kk L_{\mathbf0}$.  Hence
		$\operatorname{rank}_S L=\dim_\kk L_{\mathbf0}=\dim_\kk N_{\mathbf0}$.
		Theorem~\ref{thm:finite-intrinsic} identifies
		$\mathcal H_{n-1}(\mathbf t)$ with the objects having zero degree-zero
		component.  Since $\operatorname{rank}_S L=0$ if and only if $L$ is
		$S$-torsion, this proves \textup{(i)}.
	\end{proof}
	
	\begin{proof}[Proof of \textup{(ii)}]
		We check that taking the $S$-dual produces a source in the same box.
		Since $L(-\mathbf t)\in\mathcal A^+_{\mathbf t}$,
		Proposition~\ref{prop:positive-box} shows that it
		has a projective presentation whose terms are finite direct sums of
		$S(-\mathbf b)$ with $\mathbf0\le\mathbf b\le\mathbf t$.
		Shifting back gives a presentation
		$F_1\to F_0\to L\to0$ whose terms are sums of $S(\mathbf a)$ with
		$\mathbf0\le\mathbf a\le\mathbf t$.  Consequently
		\[
		L^*=\ker(F_0^*\longrightarrow F_1^*)\in\mathcal A^+_{\mathbf t},
		\]
		since each $F_j^*$ is a sum of $S(-\mathbf a)$ in that category,
		which is closed under kernels.
		
		Graded local duality, with $\omega_S=S(-\mathbf1)$, gives
		\begin{equation}\label{eq:top-graded-dual-source}
			\Dgr\Phi_n^{\mathbf t}(M)
			\cong\underline{\Hom}_S(M(-\mathbf1),S(-\mathbf1))
			\cong M^*
			\qquad(M\in\mathcal A^+_{\mathbf t}).
		\end{equation}
		Over the normal domain $S$, the dual of a finite module is reflexive.
		Thus $N\in\mathcal H_n(\mathbf t)$ implies that $L$ is reflexive.
		Conversely, if $L$ is reflexive, take $M=L^*\in\mathcal A^+_{\mathbf t}$.
		Then \eqref{eq:top-graded-dual-source} gives
		$\Dgr\Phi_n^{\mathbf t}(M)\cong L^{**}\cong L$, and graded Matlis
		duality gives $\Phi_n^{\mathbf t}(M)\cong N$.  This proves \textup{(ii)}.
	\end{proof}
	
	The top two images generate the ambient category under extensions.
	The following corollary determines when their union already contains
	every object.
	
	\begin{corollary}\label{cor:top-two-exhaustion}
		Assume $n\ge1$.  Then
		\[
		\mathcal H_{n-1}(\mathbf t)\cup\mathcal H_n(\mathbf t)
		=\mathcal A^-_{\mathbf t}
		\quad\Longleftrightarrow\quad \mathbf t=\mathbf0.
		\]
	\end{corollary}
	
	\begin{proof}
		If $\mathbf t=\mathbf0$, every object of $\mathcal A^+_{\mathbf0}$
		is a finite direct sum of copies of $S$.  Under the finite-box equivalence,
		$H^n_{\mathfrak m}(S^{\oplus r}(-\mathbf1))$ corresponds to $\kk^r$
		at the unique vertex of $P_{\mathbf0}$.  Hence
		$\mathcal H_n(\mathbf0)=\mathcal A^-_{\mathbf0}$, proving sufficiency.
		
		Conversely, suppose $\mathbf t\ne\mathbf0$ and choose a nonzero vertex
		$\mathbf a\in P_{\mathbf t}$.  Set
		$N=\mathcal E_{\mathbf0}\oplus\sigma_{\mathbf t}S_{\mathbf a}$.
		Then
		\[
		\Dgr N\cong S\oplus\Dgr(\sigma_{\mathbf t}S_{\mathbf a}).
		\]
		Since $\sigma_{\mathbf t}S_{\mathbf a}\in\mathcal H_{n-1}(\mathbf t)$,
		Theorem~\ref{thm:reflexive-top-complement}\textup{(i)} shows that the
		second summand is a nonzero $S$-torsion module.  Thus $\Dgr N$ is
		neither $S$-torsion nor reflexive.
		By Theorem~\ref{thm:reflexive-top-complement}\textup{(i),(ii)},
		$N$ belongs to neither $\mathcal H_{n-1}(\mathbf t)$ nor
		$\mathcal H_n(\mathbf t)$, proving necessity.
	\end{proof}
	
	\section{Applications: attached primes and Bass numbers}
	\label{sec:attached-bass}
	
	In this section, we determine the attached primes over $S$ of objects
	$L\in\mathcal H_i(\mathbf t)$ for $0\le i\le n$ and characterize the sets
	that can occur.  For $0\le i<n$, these primes are read off from the minimal
	projective cover of $L$ in $\mathcal H_i(\mathbf t)$.  Using graded Matlis
	duality, we also compute the Bass numbers at $\mathfrak m$ and the injective
	dimensions over $S$ of projective and injective objects in
	$\mathcal H_i(\mathbf t)$ for $0\le i<n$.
	
	Throughout this section, fix $\mathbf t\in\mathbb N^n$.
	For $0\le i<n$, put $q=n-i$.
	
	\begin{notation}
		For $\mathbf a\in U_q(\mathbf t)$, let
		\[
		\mathcal P_{\mathbf a}^{\mathbf t}
		=\sigma_{\mathbf t}P_{\mathbf a}^{\mathbf t}
		\]
		denote the indecomposable projective object of $\mathcal H_i(\mathbf t)$
		corresponding to the principal upper representation $P_{\mathbf a}^{\mathbf t}$,
		which has value $\kk$ at each $\mathbf b\ge\mathbf a$ and zero elsewhere,
		with identity maps between nonzero components.
		Let $I_{\mathbf a}^{\mathbf t,q}$ be the principal lower representation of
		$U_q(\mathbf t)$, extended by zero to $P_{\mathbf t}$.  Its support is
		$\{\mathbf b\in U_q(\mathbf t):\mathbf b\le\mathbf a\}$.  Write
		\[
		\mathcal I_{\mathbf a}
		=\sigma_{\mathbf t}I_{\mathbf a}^{\mathbf t,q}
		\]
		for the indecomposable injective object of $\mathcal H_i(\mathbf t)$.  Put
		\[
		F(\mathbf a)=\{j\in[n]:a_j>0\},
		\qquad
		\mathfrak p_{\mathbf a}=(x_j:j\in F(\mathbf a)),
		\qquad
		r=|F(\mathbf a)|.
		\]
		When $q$ is fixed, we also write $s=r-q$.
	\end{notation}
	
	\subsection{Attached primes}
	\label{sec:attached-primes}
	
	We use graded Matlis duality in the form
	\[
	\operatorname{Att}_S L=\operatorname{Ass}_S\Dgr L
	\]
	for the Artinian graded modules occurring here.  The first calculation is the
	graded dual of an indecomposable projective.
	
	\begin{lemma}
		\label{lem:attached-projective-dual}
		For $\mathbf a\in U_q(\mathbf t)$, set
		\[
		J_{\mathbf a}=(x_j^{t_j-a_j+1}:j\in F(\mathbf a)).
		\]
		Then
		\[
		\Dgr\mathcal P_{\mathbf a}^{\mathbf t}
		\cong (S/J_{\mathbf a})(\mathbf t),
		\qquad
		\operatorname{Att}_S\mathcal P_{\mathbf a}^{\mathbf t}
		=\{\mathfrak p_{\mathbf a}\}.
		\]
	\end{lemma}
	
	\begin{proof}
		The negative reconstruction of the principal upper representation gives
		\[
		\Dgr\mathcal P_{\mathbf a}^{\mathbf t}
		\cong (S/J_{\mathbf a})(\mathbf t).
		\]
		The ideal $J_{\mathbf a}$ is $\mathfrak p_{\mathbf a}$-primary, because it
		is generated by powers of precisely the variables in $F(\mathbf a)$.
		The equality of attached and associated primes under graded Matlis duality
		therefore gives the second assertion.
	\end{proof}
	
	The preceding lemma computes the attached primes of the projective
	objects.  The next proposition shows that passing from a projective cover
	to its quotient preserves the entire attached-prime set.
	
	\begin{proposition}
		\label{prop:attached-projective-cover}
		Let $L\in\mathcal H_i(\mathbf t)$ and let
		\[
		Q_0=\bigoplus_{\mathbf a\in U_q(\mathbf t)}
		(\mathcal P_{\mathbf a}^{\mathbf t})^{\mu_{\mathbf a}}
		\longrightarrow L
		\]
		be its projective cover in $\mathcal H_i(\mathbf t)$.  Then
		\[
		\operatorname{Att}_S L
		=\operatorname{Att}_S Q_0
		=\{\mathfrak p_{\mathbf a}:\mu_{\mathbf a}>0\}.
		\]
		Thus the attached primes are read off from the zeroth term of the minimal
		projective resolution of $L$.
	\end{proposition}
	
	\begin{proof}
		Take a minimal projective presentation
		\[
		Q_1\xrightarrow{d_1}Q_0\longrightarrow L\longrightarrow0.
		\]
		After applying $\Dgr$, one obtains
		\[
		0\longrightarrow\Dgr L\longrightarrow\Dgr Q_0
		\xrightarrow{\Dgr d_1}\Dgr Q_1.
		\]
		The lemma gives
		\[
		\operatorname{Ass}_S\Dgr Q_0
		=\{\mathfrak p_{\mathbf a}:\mu_{\mathbf a}>0\},
		\]
		so the inclusion
		\[
		\operatorname{Att}_S L\subseteq
		\{\mathfrak p_{\mathbf a}:\mu_{\mathbf a}>0\}
		\]
		is immediate.
		
		For the reverse inclusion, fix a summand indexed by $\mathbf a$ in $Q_0$ and
		localize at $\mathfrak p_{\mathbf a}$.  The socle of
		$(S/J_{\mathbf a})_{\mathfrak p_{\mathbf a}}$ is generated by the class of
		\[
		u_{\mathbf a}=\prod_{j\in F(\mathbf a)}x_j^{t_j-a_j}.
		\]
		Every nonzero component of $d_1$ mapping to the chosen summand comes from a
		summand indexed by $\mathbf b\ge\mathbf a$; minimality excludes
		$\mathbf b=\mathbf a$.  After duality, such a component is a scalar multiple
		of the natural quotient $S/J_{\mathbf a}\to S/J_{\mathbf b}$, with the
		common shift $\mathbf t$.  If $F(\mathbf b)\supsetneq F(\mathbf a)$, its target
		vanishes after localization at $\mathfrak p_{\mathbf a}$.  Otherwise some
		$b_j>a_j$ with $j\in F(\mathbf a)$, and hence $u_{\mathbf a}\in J_{\mathbf b}$.
		In either case the socle element lies in the kernel of the localized dual
		differential.  Consequently
		$(\Dgr L)_{\mathfrak p_{\mathbf a}}$ has nonzero socle, so
		$\mathfrak p_{\mathbf a}\in\operatorname{Ass}_S\Dgr L$.
	\end{proof}
	
	We can now describe all attached-prime sets that occur in the image
	categories.
	
	\begin{corollary}
		\label{cor:possible-attached-primes}
		Let $T=\{j\in[n]:t_j>0\}$.  For $0\le i<n$, set
		\[
		\mathcal C_i(\mathbf t)
		=\{(x_j:j\in F):F\subseteq T,\ |F|\ge n-i\}.
		\]
		Then
		\[
		\{\operatorname{Att}_S L:L\in\mathcal H_i(\mathbf t)\}
		=\{\mathcal Q:\mathcal Q\subseteq\mathcal C_i(\mathbf t)\}.
		\]
		Every such set is realized by a projective object of
		$\mathcal H_i(\mathbf t)$.  In top degree,
		\[
		\{\operatorname{Att}_S L:L\in\mathcal H_n(\mathbf t)\}
		=\{\varnothing,\{(0)\}\}.
		\]
	\end{corollary}
	
	\begin{proof}
		The proposition gives the containment in the non-top case.  Conversely, for
		every $F\subseteq T$ with $|F|\ge n-i$, the indicator vector
		$\mathbf1_F$ lies in $U_{n-i}(\mathbf t)$.  The direct sum of the
		$\mathcal P_{\mathbf1_F}^{\mathbf t}$ over any chosen collection of such
		sets realizes the corresponding collection of primes; the empty collection
		is realized by the zero object.
		
		For $L=H^n_{\mathfrak m}(M(-\mathbf1))$, graded local duality gives
		$\Dgr L\cong\underline{\operatorname{Hom}}_S(M,S)$.  If $L\ne0$, this is a
		nonzero finitely generated torsion-free module over the domain $S$, and hence
		has associated-prime set $\{(0)\}$.  The module
		$H^n_{\mathfrak m}(S(-\mathbf1))$ is nonzero, which proves the top-degree
		assertion.
	\end{proof}
	
	\begin{example}
		\label{ex:attached-primes}
		Let $n=3$, $\mathbf t=(1,1,1)$, and $i=1$, so $q=2$.  Take
		\[
		L=\mathcal P_{(1,1,0)}^{\mathbf t}
		\oplus\mathcal P_{(1,1,1)}^{\mathbf t}.
		\]
		Then Proposition~\ref{prop:attached-projective-cover} gives
		\[
		\operatorname{Att}_S L
		=\{(x_1,x_2),(x_1,x_2,x_3)\}.
		\]
		This illustrates that the attached-prime set records the support ranks of the
		projective summands in the projective cover.
	\end{example}
	
	\subsection{Bass numbers of projective and injective objects}
	\label{sec:bass-numbers}
	
	For the definition and basic properties of Bass numbers, see
	\cite[Section~3.2]{BH}.  For $L\in\mathcal H_i(\mathbf t)$, including
	the top degree, $L$ is Artinian
	and $\Dgr L$ is finitely generated.  Graded Matlis duality gives
	\[
	\mu^p(\mathfrak m,L)
	=\dim_{\kk}\operatorname{Ext}_S^p(\kk,L)
	=\beta_p^S(\Dgr L).
	\]
	For the projective and injective objects in non-top degrees, the
	Bass-number problem therefore reduces to computing minimal free resolutions
	of their graded duals.  The projective duals were identified in
	Lemma~\ref{lem:attached-projective-dual}; we now treat the injective duals.
	
	\begin{lemma}
		\label{lem:injective-dual-ring}
		Let $\mathbf a\in U_q(\mathbf t)$ and put $F=F(\mathbf a)$,
		$r=|F|$, and $s=r-q$.  For
		$G\subseteq F$, write
		\[
		\mathbf a_G=\sum_{j\in G}a_j\mathbf e_j,
		\qquad
		J_{\mathbf a,q}=(x^{\mathbf a_G}:G\subseteq F,\ |G|=s+1),
		\qquad
		R_{\mathbf a,q}=S/J_{\mathbf a,q}.
		\]
		Then
		\[
		\Dgr\mathcal I_{\mathbf a}\cong R_{\mathbf a,q}(\mathbf a).
		\]
		The ring $R_{\mathbf a,q}$ is Cohen--Macaulay of dimension $i$.  Its
		minimal free resolution has zeroth term $S$ and $p$th term
		\[
		\bigoplus_{\substack{G\subseteq F\\|G|=s+p}}
		S(-\mathbf a_G)^{\binom{s+p-1}{s}},
		\qquad 1\le p\le q,
		\]
		and has no terms in higher homological degrees.
	\end{lemma}
	
	\begin{proof}
		By negative reconstruction, $(\Dgr\mathcal I_{\mathbf a})_{\mathbf c}$ is
		nonzero exactly when $\mathbf c\ge-\mathbf a$ and at least $q$ coordinates
		$c_j$ with $j\in F$ are negative.  Writing $\mathbf d=\mathbf c+\mathbf a$,
		this says that $\mathbf d\ge0$ and at most $s$ indices $j\in F$ satisfy
		$d_j\ge a_j$.  This is precisely the condition that $x^{\mathbf d}$ survives
		modulo $J_{\mathbf a,q}$, proving the displayed duality.
		
		Let $T=\kk[y_j:j\in F]$ and let $V_{s+1}$ be generated by all squarefree
		monomials of degree $s+1$ in $T$.  The quotient is the Stanley--Reisner ring
		of the simplex skeleton
		$\Delta=\{G\subseteq F:|G|\le s\}$.  For $|G|>s$, the induced complex on
		$G$ has reduced homology only in degree $s-1$, of dimension
		$\binom{|G|-1}{s}$; for $0<|G|\le s$ it is a simplex.  When $s=0$, use
		reduced homology in degree $-1$ of $\{\varnothing\}$.  For $G\subseteq F$,
		we abbreviate $\beta^T_{p,\mathbf1_G}$ to $\beta^T_{p,G}$, where
		$\mathbf1_G\in\{0,1\}^F$ is the indicator vector of $G$; in particular,
		$\mathbf1_{\varnothing}=\mathbf0$.  Hochster's formula therefore gives,
		for $p\ge1$,
		\[
		\beta^T_{p,G}(T/V_{s+1})
		=\begin{cases}
			\binom{|G|-1}{s},&|G|=s+p,\\
			0,&\text{otherwise}.
		\end{cases}
		\]
		In homological degree zero, $\beta^T_{0,\varnothing}(T/V_{s+1})=1$
		and all other multigraded Betti numbers vanish.
		The substitution $y_j\mapsto x_j^{a_j}$ is finite free, and adjoining inactive
		variables is flat.  These operations preserve minimality and give the stated
		resolution of $R_{\mathbf a,q}$.  Its length is $q$, while the radical of
		$J_{\mathbf a,q}$ has height $q$.  Hence
		$\depth R_{\mathbf a,q}=\dim R_{\mathbf a,q}=n-q=i$.
	\end{proof}
	
	Taking the ranks in these resolutions gives the Bass numbers.
	
	\begin{proposition}
		\label{prop:bass-projectives-injectives}
		Let $\mathbf a\in U_q(\mathbf t)$, and write $r=\ell(\mathbf a)$ and
		$s=r-q$.  Then
		\[
		\mu^p(\mathfrak m,\mathcal P_{\mathbf a}^{\mathbf t})
		=\binom rp\qquad(p\ge0),
		\]
		where $\binom rp=0$ for $p>r$, and
		\[
		\mu^p(\mathfrak m,\mathcal I_{\mathbf a})
		=\begin{cases}
			1,&p=0,\\[2pt]
			\displaystyle\binom r{s+p}\binom{s+p-1}{s},&1\le p\le q,\\[4pt]
			0,&p>q.
		\end{cases}
		\]
		Bass numbers of arbitrary projective or injective objects are obtained by
		adding these formulas with the multiplicities of their indecomposable summands.
	\end{proposition}
	
	\begin{proof}
		The first formula follows from
		\[
		\Dgr\mathcal P_{\mathbf a}^{\mathbf t}
		\cong\bigl(S/(x_j^{t_j-a_j+1}:j\in F(\mathbf a))\bigr)(\mathbf t).
		\]
		The displayed pure powers form a regular sequence of length $r$, so their
		Koszul resolution has ranks $\binom rp$.
		
		For the injective object, apply the preceding lemma and take the ranks in its
		minimal free resolution.  Finally, Bass numbers are additive on finite direct
		sums.
	\end{proof}
	
	The last nonzero Bass numbers determine the injective dimensions over $S$.
	
	\begin{corollary}
		\label{cor:target-injective-dimensions}
		For $0\le i<n$ and $\mathbf a\in U_{n-i}(\mathbf t)$,
		\[
		\operatorname{id}_S\mathcal P_{\mathbf a}^{\mathbf t}=\ell(\mathbf a),
		\qquad
		\operatorname{id}_S\mathcal I_{\mathbf a}=n-i.
		\]
		Consequently, if
		$P=\bigoplus_{\mathbf a\in U_{n-i}(\mathbf t)}(\mathcal P_{\mathbf a}^{\mathbf t})^{m_{\mathbf a}}$
		is a nonzero projective object of $\mathcal H_i(\mathbf t)$, then
		\[
		\operatorname{id}_S P
		=\max\{\ell(\mathbf a)\mid m_{\mathbf a}>0\}.
		\]
		Every nonzero injective object $I$ of $\mathcal H_i(\mathbf t)$ satisfies
		\[
		\operatorname{id}_S I=n-i.
		\]
	\end{corollary}
	
	\begin{proof}
		Since $\Dgr S$ is the injective hull of $S/\mathfrak m$ as an ungraded
		$S$-module, the graded duals of the minimal free resolutions above give
		injective resolutions over $S$.  Their lengths are $r=\ell(\mathbf a)$
		and $q=n-i$, respectively.  The last Bass numbers are nonzero, since
		\[
		\mu^r(\mathfrak m,\mathcal P_{\mathbf a}^{\mathbf t})=1,
		\qquad
		\mu^q(\mathfrak m,\mathcal I_{\mathbf a})
		=\binom{r-1}{q-1}>0.
		\]
		This proves the equalities; the assertions for finite direct sums follow
		immediately.
	\end{proof}
	
	\begin{example}
		\label{ex:bass-numbers}
		Let $n=3$, $\mathbf t=(1,1,1)$, $i=1$, and
		$\mathbf a=(1,1,1)$.  Then $q=2$, $r=3$, and $s=1$.  Hence
		\[
		\bigl(\mu^p(\mathfrak m,\mathcal P_{\mathbf a}^{\mathbf t})\bigr)_{p\ge0}
		=(1,3,3,1,0,\ldots),
		\]
		whereas
		\[
		\bigl(\mu^p(\mathfrak m,\mathcal I_{\mathbf a})\bigr)_{p\ge0}
		=(1,3,2,0,\ldots).
		\]
		The two sequences come from the Koszul resolution of the projective dual and
		the squarefree-skeleton resolution of the injective dual, respectively.
	\end{example}
	
	\section{Functorial sections of local-cohomology functors}
	\label{sec:functorial-section}
	
	The classifications in Section 4 describe which targets occur.  In this section, we  ask whether
	sources and their morphisms can be chosen functorially.
	Using functorial free resolutions over $\Lambda_{\mathbf t}$, we construct
	an exact functorial section of $\Phi_i^{\mathbf t}$ for every $0\le i<n$.
	For each nonzero object of $\mathcal H_i(\mathbf t)$, the selected source
	has minimal projective dimension $n-i$ and maximal depth $i$ among all
	its sources.  We also determine the possible projective dimensions of
	sources when $\mathbf t\ge\mathbf1$.  The constructed sections are
	faithful but not full whenever their domains are nonzero.  In degree
	zero, we compare this construction with the fully faithful shift section
	$L\mapsto L(\mathbf1)$.  We then show that, for $n\ge2$ and
	$\mathbf t\ge\mathbf1$, no fully faithful section exists
	for $i=1$ or $i=n-1$, even without requiring exactness, and explain why
	the arguments do not extend directly to the intermediate degrees.
	We conclude by constructing a fully faithful section in top degree
	using graded Matlis duality and $S$-duality.
	
	\subsection{Sections in non-top degrees}
	
	\begin{definition}\label{6.1}
		Let $F:\mathcal C\to\mathcal D$ be a functor.  A \emph{functorial
			section} of $F$ is a functor $s:\mathcal D\to\mathcal C$ together with a
		natural isomorphism
		\[
		\eta:F\circ s\xRightarrow{\sim}\operatorname{id}_{\mathcal D}.
		\]
		When $\mathcal C$ and $\mathcal D$ are abelian, the section is called
		\emph{exact} if $s$ is exact.  A section is a quasi-inverse only when it also satisfies
		$s\circ F\cong\operatorname{id}_{\mathcal C}$.
	\end{definition}
	
	We remark that under the finite-box equivalences, extension by zero from
	$U_{n-i}(\mathbf t)$ to $P_{\mathbf t}$ identifies
	$\mathcal H_i(\mathbf t)$ with a Serre subcategory of
	$\mathcal A^+_{\mathbf t}$.  However, this embedding is not a section of
	$\Phi_i^{\mathbf t}$ unless $\mathcal H_i(\mathbf t)=0$
	(i.e., $\ell(\mathbf t)<n-i$).
	
	Exactness is a separate requirement: the section identity in
	Definition~\ref{6.1} does not imply it, even when the original functor
	is exact, as the following example shows.
	
	\begin{example}
		Let $\mathcal A$ and $\mathcal B$ be abelian categories and let
		$F:\mathcal A\to\mathcal B$ be an additive functor which is not exact.
		Consider the projection
		\[
		\pi:\mathcal A\times\mathcal B\longrightarrow\mathcal A,
		\qquad (X,Y)\longmapsto X,
		\]
		and define
		\[
		s:\mathcal A\longrightarrow\mathcal A\times\mathcal B,
		\qquad s(X)=(X,F(X)),\quad s(f)=(f,F(f)).
		\]
		Then $\pi\circ s=\operatorname{id}_{\mathcal A}$, so $s$ is an additive
		section of the exact functor $\pi$.  Since exactness in
		$\mathcal A\times\mathcal B$ is determined componentwise, $s$ is exact
		if and only if $F$ is exact.  Thus $s$ is not exact.
	\end{example}
	
	We construct the desired exact sections by making the realization
	argument in Lemma~\ref{lem:realization} functorial. We begin with fixing notation.
	
	\begin{notation}
		Fix $0\le i<n$ and put $q=n-i$.  Set $A=\Lambda_{\mathbf t}$, and let
		\[
		\kappa_{\mathbf t}:A\text{-}\modcat\longrightarrow\mathcal A^+_{\mathbf t}
		\]
		denote the quasi-inverse of $\pi_{\mathbf t}$ after the identification
		$A\text{-}\modcat\simeq\Rep(P_{\mathbf t})$.  We use
		\[
		(-)^*=\operatorname{Hom}_{A^{\mathrm{op}}}(-,A).
		\]
	\end{notation}
	
	We use the  functorial free resolution constructed below.  Since $A$ is
	finite-dimensional over $\kk$, a right $A$-module is finite-dimensional
	if and only if it is finitely generated.  For such a module $X$, put
	$Z_0(X)=X$ and define
	\[
	P_0(X)=Z_0(X)\otimes_\kk A,
	\qquad
	\varepsilon_0:P_0(X)\longrightarrow Z_0(X),
	\quad x\otimes a\longmapsto xa,
	\]
	and set $Z_1(X)=\ker\varepsilon_0$.  Inductively, once $Z_r(X)$ has been
	defined, put
	\[
	P_r(X)=Z_r(X)\otimes_\kk A,
	\qquad
	\varepsilon_r:P_r(X)\longrightarrow Z_r(X),
	\quad z\otimes a\longmapsto za,
	\]
	and set $Z_{r+1}(X)=\ker\varepsilon_r$.
	Since $Z_0(X)=X$ and $Z_{r+1}(X)\subseteq Z_r(X)\otimes_\kk A$,
	induction shows that every $Z_r(X)$ is finite-dimensional.
	Hence each $P_r(X)$ is a finitely generated free right $A$-module.
	For $r\ge1$, let
	\[
	d_r:P_r(X)\xrightarrow{\varepsilon_r}Z_r(X)
	\lhook\joinrel\longrightarrow P_{r-1}(X).
	\]
	Thus
	\[
	\cdots\longrightarrow P_2(X)\xrightarrow{d_2}P_1(X)
	\xrightarrow{d_1}P_0(X)\xrightarrow{\varepsilon_0}X\longrightarrow0
	\]
	is a free resolution of $X$.
	
	A homomorphism $f:X\to X'$ induces $Z_0(f)=f$ and, inductively,
	$P_r(f)=Z_r(f)\otimes_\kk\operatorname{id}_A$.  The identity
	\[
	\varepsilon_r^{X'}\circ P_r(f)
	=Z_r(f)\circ\varepsilon_r^X
	\]
	shows that $P_r(f)$ restricts to a map
	$Z_{r+1}(f):Z_{r+1}(X)\to Z_{r+1}(X')$.  Hence the construction gives a
	functorial chain map lifting every $f$.
	
	Define
	\[
	\mathsf R_q(X)=\operatorname{coker}\!\left(
	P_{q-1}(X)^*\xrightarrow{d_q^*}P_q(X)^*
	\right).
	\]
	If $\grade_{A^{\mathrm{op}}}X\ge q$, then
	$\operatorname{Ext}_{A^{\mathrm{op}}}^r(X,A)=0$ for $r<q$, and the dual
	complex has an exact truncation
	\[
	0\longrightarrow P_0(X)^*\longrightarrow P_1(X)^*\longrightarrow\cdots
	\longrightarrow P_q(X)^*\longrightarrow\mathsf R_q(X)\longrightarrow0.
	\]
	Applying $\operatorname{Hom}_A(-,A)$ gives a natural isomorphism
	\[
	\bigl(\operatorname{Ext}_A^q(-,A)\circ\mathsf R_q\bigr)(X)\cong X.
	\]
	The construction $\mathsf R_q$ is contravariant in $X$.
	
	\begin{proposition}
		\label{prop:Rq-exact}
		The functor $\mathsf R_q$ is contravariantly exact on $\mathcal G_q(A)$.
	\end{proposition}
	
	\begin{proof}
		Let
		\[
		0\longrightarrow X'\longrightarrow X\longrightarrow X''\longrightarrow0
		\]
		be a short exact sequence in $\mathcal G_q(A)$.  We first show that the
		functorial free resolutions above form a degreewise short exact sequence.
		At stage $r=0$ this is the given sequence.  Suppose inductively that
		\[
		0\longrightarrow Z_r(X')\longrightarrow Z_r(X)
		\longrightarrow Z_r(X'')\longrightarrow0
		\]
		is exact.  Tensoring over the field $\kk$ with $A$ gives a short exact
		sequence of the modules $P_r$.  Naturality of the augmentations and the
		snake lemma then give
		\[
		0\longrightarrow Z_{r+1}(X')\longrightarrow Z_{r+1}(X)
		\longrightarrow Z_{r+1}(X'')\longrightarrow0.
		\]
		Induction yields a short exact sequence of complexes
		\[
		0\longrightarrow P_\bullet(X')\longrightarrow P_\bullet(X)
		\longrightarrow P_\bullet(X'')\longrightarrow0.
		\]
		Each degreewise sequence splits as a sequence of right $A$-modules because
		$P_r(X'')$ is free.  Applying $(-)^*$ therefore gives a degreewise short
		exact sequence of cochain complexes in the reverse order.
		
		For $Y\in\{X',X,X''\}$, let $C_q(Y)$ be the complex obtained from
		$P_\bullet(Y)^*$ by retaining degrees $0,\ldots,q$ and setting all other
		terms to zero.  Its differential out of degree $q$ is zero.  Since
		$\grade_{A^{\mathrm{op}}}Y\ge q$,
		\[
		(H^{q-1}\circ C_q)(Y)
		=\operatorname{Ext}_{A^{\mathrm{op}}}^{q-1}(Y,A)=0.
		\]
		For $q=1$ this is the vanishing of
		$\operatorname{Hom}_{A^{\mathrm{op}}}(Y,A)$.  By construction,
		\[
		(H^q\circ C_q)(Y)
		=\operatorname{coker}\bigl(P_{q-1}(Y)^*\longrightarrow P_q(Y)^*\bigr)
		=\mathsf R_q(Y),
		\qquad (H^{q+1}\circ C_q)(Y)=0.
		\]
		The long exact cohomology sequence of
		\[
		0\longrightarrow C_q(X'')\longrightarrow C_q(X)
		\longrightarrow C_q(X')\longrightarrow0
		\]
		therefore contains
		\[
		0\longrightarrow\mathsf R_q(X'')\longrightarrow\mathsf R_q(X)
		\longrightarrow\mathsf R_q(X')\longrightarrow0.
		\]
		These are the maps induced by the functorial resolutions, which proves the
		assertion.
	\end{proof}
	
	We now transport this construction back to the categories of determined
	modules.
	
	\begin{theorem}
		\label{thm:functorial-section}
		For every $0\le i<n$, the formula
		\[
		s_i^{\mathbf t}(L)=
		\bigl(\kappa_{\mathbf t}\circ\mathsf R_{n-i}\circ D
		\circ\theta_{\mathbf t}\bigr)(L),
		\qquad L\in\mathcal H_i(\mathbf t)
		\]
		defines an additive exact functor
		\[
		s_i^{\mathbf t}:\mathcal H_i(\mathbf t)
		\longrightarrow\mathcal A^+_{\mathbf t}
		\]
		with a natural isomorphism
		\[
		\boxed{\quad
			\Phi_i^{\mathbf t}\circ s_i^{\mathbf t}
			\cong\operatorname{id}_{\mathcal H_i(\mathbf t)}.
			\quad}
		\]
		If $\mathcal H_i(\mathbf t)\ne0$, then $s_i^{\mathbf t}$ is faithful but
		not full.
	\end{theorem}
	
	\begin{proof}
		Let $L\in\mathcal H_i(\mathbf t)$ and put
		$V=\theta_{\mathbf t}(L)$.  By the finite-box classification,
		$V\in\Rep(U_{n-i}(\mathbf t))$.  Under the duality between representations
		and right $A$-modules, $X=D(V)$ satisfies
		\[
		\grade_{A^{\mathrm{op}}}X\ge n-i=q.
		\]
		The preceding construction consequently gives
		\[
		\bigl(\operatorname{Ext}_A^q(-,A)\circ\mathsf R_q\bigr)(X)\cong X,
		\]
		naturally in $L$.  Proposition~\ref{prop:naka-ext} now yields
		\[
		\begin{aligned}
			(\theta_{\mathbf t}\circ\Phi_i^{\mathbf t}\circ s_i^{\mathbf t})(L)
			&\cong
			\bigl(D\circ\operatorname{Ext}_A^q(-,A)\circ\mathsf R_q\bigr)(X)\\
			&\cong DX
			\cong (D\circ D)(V)
			\cong V.
		\end{aligned}
		\]
		Since $\theta_{\mathbf t}$ is an equivalence, this proves the section
		identity.  Moreover, $\theta_{\mathbf t}$ and $\kappa_{\mathbf t}$ are
		exact equivalences, $D$ is exact and contravariant, and
		Proposition~\ref{prop:Rq-exact} applies on the grade subcategory containing all such
		$X$.  Hence $s_i^{\mathbf t}$ is exact.
		
		The section identity implies faithfulness.  To prove nonfullness, take
		$L\ne0$.  Then $X\ne0$, and $A$ has at least two vertex idempotents.
		The construction gives
		$\dim_\kk Z_r(X)=(\dim_\kk A-1)^r\dim_\kk X>0$.
		Choose a vertex idempotent $e$ such that
		$W=Z_q(X)(1-e)\ne0$.  The right projective module
		$C=W\otimes_\kk eA$ is a nonzero direct summand of
		$P_q(X)=Z_q(X)\otimes_\kk A$, and $\varepsilon_q(C)=0$.
		Thus $d_q(C)=0$, so $\mathsf R_q(X)$ has the nonzero projective direct
		summand $C^*$.  Projection onto this summand induces a nonzero
		endomorphism of $s_i^{\mathbf t}(L)$ killed by $\Phi_i^{\mathbf t}$,
		since $q>0$.  The section identity shows that this endomorphism cannot
		lie in the image of $s_i^{\mathbf t}$.  Hence $s_i^{\mathbf t}$ is not full.
	\end{proof}
	
	The chosen sources also attain the smallest possible projective
	dimension and the largest possible depth.
	
	\begin{corollary}
		\label{cor:section-optimality}
		Let $0\ne L\in\mathcal H_i(\mathbf t)$ with $0\le i<n$.  Then
		\begin{enumerate}[label=\textup{(\arabic*)},leftmargin=*,itemsep=6pt]
			\item $\displaystyle
			\pd_{\mathcal A^+_{\mathbf t}}s_i^{\mathbf t}(L)
			=\pd_S s_i^{\mathbf t}(L)=n-i
			=\min\{\pd_S N\mid N\in\mathcal A^+_{\mathbf t},\
			\Phi_i^{\mathbf t}(N)\cong L\}$.
			\item $\displaystyle
			\depth_S s_i^{\mathbf t}(L)
			=i
			=\max\{\depth_S N\mid N\in\mathcal A^+_{\mathbf t},\
			\Phi_i^{\mathbf t}(N)\cong L\}$.
		\end{enumerate}
	\end{corollary}
	
	\begin{proof}
		Put $q=n-i$ and $X=(D\circ\theta_{\mathbf t})(L)$.  The defining construction
		provides a projective resolution
		\[
		0\longrightarrow P_0(X)^*\longrightarrow\cdots\longrightarrow P_q(X)^*
		\longrightarrow\mathsf R_q(X)\longrightarrow0.
		\]
		Transporting it through $\kappa_{\mathbf t}$ gives
		$\pd_{\mathcal A^+_{\mathbf t}}s_i^{\mathbf t}(L)\le q$.  For any source
		$N$ of $L$, Proposition~\ref{prop:naka-ext} gives
		\[
		\bigl(D\circ\Ext^q_{\Lambda_{\mathbf t}}(-,\Lambda_{\mathbf t})
		\circ\pi_{\mathbf t}\bigr)(N)
		\cong\theta_{\mathbf t}L\ne0.
		\]
		Thus $\pd_{\mathcal A^+_{\mathbf t}}N\ge q$.  Proposition~\ref{prop:box-homological-dimensions}\textup{(ii)} identifies this dimension with
		$\pd_S N$, and the Auslander--Buchsbaum formula gives
		$\depth_S N=n-\pd_S N$.  This proves the two equalities.
	\end{proof}
	
	\begin{remark}
		\label{rem:source-pd-range}
		If $\mathbf t\ge\mathbf1$ and $0\ne L\in\mathcal H_i(\mathbf t)$, then
		\[
		\bigl\{\pd_S N:N\in\mathcal A^+_{\mathbf t},\
		\Phi_i^{\mathbf t}(N)\cong L\bigr\}
		=\{n-i,n-i+1,\ldots,n\}.
		\]
		Indeed, the corollary gives the minimum $n-i$.  For $n-i<d\le n$, the
		module $C_d=S/(x_1,\ldots,x_d)$ belongs to $\mathcal A^+_{\mathbf t}$,
		has projective dimension $d$, and has Krull dimension $n-d<i$.  Hence
		$\Phi_i^{\mathbf t}(C_d)=0$, and
		$s_i^{\mathbf t}(L)\oplus C_d$ realizes the value $d$.  No other values
		occur because $\operatorname{gl.dim}S=n$.
	\end{remark}
	
	\begin{remark}
		When $i=0$, the theorem gives
		\[
		s_0^{\mathbf t}(L)=
		\bigl(\kappa_{\mathbf t}\circ\mathsf R_n\circ D
		\circ\theta_{\mathbf t}\bigr)(L),
		\qquad
		\Phi_0^{\mathbf t}\circ s_0^{\mathbf t}
		\cong\operatorname{id}_{\mathcal H_0(\mathbf t)}.
		\]
		There is also the simpler shift section
		\[
		s_{0,\mathrm{shift}}^{\mathbf t}(L):=L(\mathbf1).
		\]
		Since $L$ is $\mathfrak m$-torsion and $L(\mathbf1)$ belongs to
		$\mathcal A^+_{\mathbf t}$,
		\[
		\Phi_0^{\mathbf t}(L(\mathbf1))
		=H^0_{\mathfrak m}(L)\cong L.
		\]
		The section identity gives
		\[
		(H^0_{\mathfrak m}\circ s_0^{\mathbf t})(L)\cong L(\mathbf1).
		\]
		Thus the shift section recovers the torsion submodule of the general source;
		the latter can contain additional free summands.
	\end{remark}
	
	\begin{example}\label{ex:degree-zero-one-variable}
		Let $S=\kk[x]$, $t=1$, and $L=\kk(-1)$.  Then
		$L\in\mathcal H_0(1)$ and $L(1)=\kk$.  The incidence algebra
		$A=\Lambda_{(1)}$ has basis $e_0=e_{0,0}$, $e_1=e_{1,1}$, and
		$u=e_{1,0}$.  The right module $X=(D\circ\theta_{(1)})(L)$ is the simple $R_1$.
		Identifying $P_0=X\otimes_\kk A$ with $A$, the augmentation is
		$a\mapsto va$, where $v$ is a basis vector of $X$.  Its kernel is
		$Z_1=\kk e_0\oplus\kk u$.  Consequently $P_1\cong A^{\oplus2}$ and
		\[
		d_1:A^{\oplus2}\longrightarrow A,
		\qquad (a,b)\longmapsto e_0a+ub.
		\]
		Under the standard identifications of the dual free modules,
		\[
		d_1^*:A\longrightarrow A^{\oplus2},
		\qquad c\longmapsto(ce_0,cu).
		\]
		Its image is spanned by $(e_0,0)$, $(u,0)$, and $(0,u)$.  If $S_0,S_1$
		denote the left vertex simples, then
		\[
		\mathsf R_1(X)
		\cong (A/Ae_0)\oplus(A/\kk u)
		\cong S_0\oplus S_1^{\oplus2}.
		\]
		Since $\kappa_{(1)}(S_0)=\kk$ and
		$\kappa_{(1)}(S_1)=S(-1)$,
		\[
		\boxed{\quad
			s_0^{(1)}(\kk(-1))\cong \kk\oplus S(-1)^{\oplus2},
			\qquad
			s_{0,\mathrm{shift}}^{(1)}(\kk(-1))=\kk.
			\quad}
		\]
		The two sources are not isomorphic, although applying $\Phi_0^{(1)}$ to
		either gives $\kk(-1)$.  By additivity, for every $r\ge1$,
		\[
		s_0^{(1)}(\kk(-1)^{\oplus r})
		\cong \kk^{\oplus r}\oplus S(-1)^{\oplus2r},
		\qquad
		s_{0,\mathrm{shift}}^{(1)}(\kk(-1)^{\oplus r})=\kk^{\oplus r}.
		\]
	\end{example}
	
	\begin{example}
		Let $S=\kk[x]$ and $t=2$, and take
		$L=(S/(x^2))(-1)$.  The shift section gives $L(1)=S/(x^2)$.  Here
		$A=\Lambda_{(2)}$ has dimension $6$, and $X=(D\circ\theta_{(2)})(L)$ has dimension
		$2$.  Thus $P_0\cong A^{\oplus2}$, $\dim_\kk Z_1=10$, and
		$P_1\cong A^{\oplus10}$.  Since $X$ has grade at least one,
		\[
		0\longrightarrow A^{\oplus2}\longrightarrow A^{\oplus10}
		\longrightarrow\mathsf R_1(X)\longrightarrow0
		\]
		is exact.  The left regular module $A$ has dimension vector $(1,2,3)$,
		so $\mathsf R_1(X)$ has dimension vector $(8,16,24)$.
		
		Put $M=s_0^{(2)}(L)$.  The section property gives
		$H^0_{(x)}(M)\cong S/(x^2)$, with dimension vector $(1,1,0)$ in the box.
		Since a finitely generated graded torsion-free $\kk[x]$-module is graded
		free, $M/H^0_{(x)}(M)$ is graded free and the quotient sequence splits.  Its
		dimension vector in the box is $(7,15,24)$; taking successive differences
		therefore gives
		\[
		s_0^{(2)}(L)\cong
		S/(x^2)\oplus S^{\oplus7}\oplus S(-1)^{\oplus8}
		\oplus S(-2)^{\oplus9}.
		\]
		Thus the general section can contain additional free summands.
	\end{example}
	
	\subsection{Obstructions to full faithfulness}
	
	The sections $s_i^{\mathbf t}$ of Theorem~\ref{thm:functorial-section}
	are not full whenever their domains are nonzero.
	In degree zero, the shift section is fully faithful.
	The next proposition shows that for $n\ge2$ and
	$\mathbf t\ge\mathbf1$, no fully faithful section exists in degrees
	$1$ and $n-1$.
	
	\begin{proposition}
		\label{prop:no-fully-faithful-section}
		Let $n\ge2$ and let $\mathbf t\in\mathbb N^n$ satisfy
		$\mathbf t\ge\mathbf1$ coordinatewise.  If $i\in\{1,n-1\}$, then
		\[
		\Phi_i^{\mathbf t}:\mathcal A^+_{\mathbf t}
		\longrightarrow\mathcal H_i(\mathbf t)
		\]
		admits no fully faithful functorial section, even without requiring the
		section to be exact.
	\end{proposition}
	
	\begin{proof}
		\emph{Case $i=1$.}
		Identify $\mathcal H_1(\mathbf t)$ with $\Rep(U_{n-1}(\mathbf t))$
		via $\theta_{\mathbf t}$, and let $T_{\mathbf a}$ denote the vertex
		simple at $\mathbf a$.  For $1\le j\le n$, put
		\[
		E_j=S/(x_\ell^{t_\ell}:\ell\in[n]\setminus\{j\}),\qquad
		\mathbf c_j=\mathbf t-\mathbf1+\mathbf e_j,
		\qquad
		J_j=(\theta_{\mathbf t}\circ\Phi_1^{\mathbf t})(E_j).
		\]
		The representation $\pi_{\mathbf t}E_j$ is the principal injective
		supported on $[\mathbf0,\mathbf c_j]$.  In particular, $E_j$ is
		injective in $\mathcal A^+_{\mathbf t}$, and for every
		$N\in\mathcal A^+_{\mathbf t}$ there is an isomorphism
		\begin{equation}\label{eq:no-full-section-injective-hom}
			\Hom_{\mathcal A^+_{\mathbf t}}(N,E_j)\cong D(N_{\mathbf c_j}).
		\end{equation}
		Since localizing $E_j$ at any $x_\ell$ with $\ell\ne j$ gives zero,
		its \v{C}ech complex shows that $J_j$ equals $\kk$ at precisely the vertices
		\[
		a_j=0,\qquad 1\le a_\ell\le t_\ell\quad(\ell\ne j),
		\]
		with identity maps between these vertices.  Thus $\End(J_j)=\kk$.
		
		We claim that every source $N$ of $J_j$ has $E_j$ as a direct summand.
		The source condition gives
		\[
		H^1_{\mathfrak m}(N)_{-\mathbf e_j}\cong\kk,
		\qquad H^1_{\mathfrak m}(N)_{\mathbf0}=0.
		\]
		In degree $-\mathbf e_j$, the zeroth term of the \v{C}ech complex is zero,
		and its only possibly nonzero term in cohomological degree one is
		$(N_{x_j})_{-\mathbf e_j}$.  Choose a nonzero cocycle $v$ in this
		space.  Multiplication by $x_j$ gives a degree-zero cocycle whose only
		nonzero component is the $j$th one.  Since
		$H^1_{\mathfrak m}(N)_{\mathbf0}=0$, there is $u\in N_{\mathbf0}$
		whose image under the zeroth \v{C}ech differential is
		\[
		(0,\ldots,0,x_jv,0,\ldots,0).
		\]
		For $\ell\ne j$, the image of $u$ in $N_{x_\ell}$ is zero, so a
		power of $x_\ell$ kills $u$.  Positive $\mathbf t$-determination then
		implies $x_\ell^{t_\ell}u=0$.  Hence $1\mapsto u$ defines a
		degree-zero map $g:E_j\to N$.  On local cohomology in degree
		$-\mathbf e_j$, it sends the class of $1/x_j$ to $v$.  Consequently,
		$(\theta_{\mathbf t}\circ\Phi_1^{\mathbf t})(g)$ is nonzero and is
		therefore an isomorphism, since its source and target are isomorphic to
		$J_j$ and $\End(J_j)=\kk$.
		
		To see that $g$ is monic, put
		$\mathbf b_j=\mathbf t-t_j\mathbf e_j$.  For every positively
		$\mathbf t$-determined module $V$, the \v{C}ech complex and positive
		determination give natural maps
		\[
		H^1_{\mathfrak m}(V)_{\mathbf b_j-\mathbf1}
		\lhook\joinrel\longrightarrow
		(V_{x_j})_{\mathbf b_j-\mathbf1}
		\cong V_{\mathbf c_j},
		\]
		where the last identification sends $w\in V_{\mathbf c_j}$ to
		$w/x_j^{t_j+1}$ in the localization.  Applying this to $E_j$ and $N$
		shows that $g_{\mathbf c_j}\ne0$, because
		$\Phi_1^{\mathbf t}(g)$ is an isomorphism and
		$(J_j)_{\mathbf b_j}=\kk$.  Every nonzero subrepresentation of
		$\pi_{\mathbf t}E_j$ is nonzero at $\mathbf c_j$, as all the structure
		maps of $\pi_{\mathbf t}E_j$ inside $[\mathbf0,\mathbf c_j]$ are
		identities.  Since
		$(E_j)_{\mathbf c_j}=\kk$, it follows that $\ker g=0$.
		Injectivity of $E_j$ in $\mathcal A^+_{\mathbf t}$ now splits $g$.
		This proves the claim and shows that $E_j$ is the unique
		indecomposable source of $J_j$.
		
		Suppose that $s$ is a fully faithful section.  Since $\End(J_j)=\kk$,
		the source $s(J_j)$ is indecomposable, and the claim gives
		$s(J_j)\cong E_j$ for every $j$.  Put $M=s(T_{\mathbf t})$.
		The section identity gives
		\[
		H^1_{\mathfrak m}(M)_{\mathbf t-\mathbf1}\cong\kk.
		\]
		The degree-one term of the \v{C}ech complex in this multidegree is
		\[
		\bigoplus_{j=1}^n(M_{x_j})_{\mathbf t-\mathbf1}
		\cong\bigoplus_{j=1}^n M_{\mathbf c_j},
		\]
		where positive determination identifies $w\in M_{\mathbf c_j}$ with
		$w/x_j$.  Hence $M_{\mathbf c_j}\ne0$ for at least one $j$.
		Equation~\eqref{eq:no-full-section-injective-hom} then gives
		$\Hom_{\mathcal A^+_{\mathbf t}}(M,E_j)\ne0$.  This contradicts full
		faithfulness, since $\Hom(T_{\mathbf t},J_j)=0$.
		
		\smallskip
		\noindent\emph{Case $i=n-1$.}
		Put $A=\Lambda_{\mathbf t}$.  By Theorem~\ref{thm:finite-intrinsic},
		identify $\mathcal H_{n-1}(\mathbf t)$ with $\Rep(U_1(\mathbf t))$.
		For $\mathbf a\in U_1(\mathbf t)$, let $T_{\mathbf a}$ be the vertex
		simple and put $R_{\mathbf a}=DT_{\mathbf a}$, viewed as a right
		$A$-module.  Suppose that $s$ is a fully faithful section, with its
		domain identified with $\Rep(U_1(\mathbf t))$.
		
		We first determine the source of $T_{\mathbf t}$.  We use the transpose
		over the finite-dimensional algebra $A$.  For a left $A$-module $X$,
		choose a minimal projective presentation
		\[
		P_1\xrightarrow{d_1}P_0\longrightarrow X\longrightarrow0
		\]
		and define
		\[
		\operatorname{Tr}_A X
		=\coker\!\left(
		\Hom_A(P_0,A)\xrightarrow{d_1^*}\Hom_A(P_1,A)
		\right).
		\]
		Extending the presentation to a projective resolution gives a natural
		monomorphism of right $A$-modules
		\[
		\Ext_A^1(X,A)\lhook\joinrel\longrightarrow\operatorname{Tr}_A X.
		\]
		Indeed, the left-hand side is
		$\ker d_2^*/\im d_1^*$, whereas the right-hand side is
		$\Hom_A(P_1,A)/\im d_1^*$.
		For an indecomposable nonprojective module $X$, the transpose is again
		indecomposable and nonprojective, and
		\[
		\bigl(\operatorname{Tr}_{A^{\mathrm{op}}}\circ\operatorname{Tr}_A\bigr)(X)\cong X;
		\]
		these are the standard properties of the transpose defined by minimal
		projective presentations; see \cite[Chapter~IV, Proposition~1.7]{ARS}.
		
		Set $N_{\mathbf t}=s(T_{\mathbf t})$ and
		$X=\pi_{\mathbf t}N_{\mathbf t}$.
		Full faithfulness makes $N_{\mathbf t}$ indecomposable, since
		$\End(T_{\mathbf t})=\kk$.
		Proposition~\ref{prop:naka-ext} and the section identity give
		\[
		\Ext_A^1(X,A)\cong R_{\mathbf t}\ne0,
		\]
		so $X$ is nonprojective.  Write $e_{\mathbf a}=e_{\mathbf a,\mathbf a}$
		for the vertex idempotents.  Since $\mathbf t$ is the largest vertex,
		$Ae_{\mathbf t}=T_{\mathbf t}$ is projective, and consequently
		$R_{\mathbf t}=D(Ae_{\mathbf t})$ is an injective right $A$-module.
		Thus the monomorphism
		$R_{\mathbf t}\hookrightarrow\operatorname{Tr}_A X$ splits.
		The indecomposability of $\operatorname{Tr}_A X$ now implies
		\[
		\operatorname{Tr}_A X\cong R_{\mathbf t},
		\qquad
		X\cong\operatorname{Tr}_{A^{\mathrm{op}}}R_{\mathbf t}.
		\]
		
		The immediate predecessors of $\mathbf t$ are precisely
		$\mathbf t-\mathbf e_j$, $1\le j\le n$.  Hence the right simple
		$R_{\mathbf t}$ has minimal projective presentation
		\[
		\bigoplus_{j=1}^n e_{\mathbf t-\mathbf e_j}A
		\longrightarrow e_{\mathbf t}A
		\longrightarrow R_{\mathbf t}\longrightarrow0,
		\]
		where the $j$th component of the first map is left multiplication by
		$e_{\mathbf t,\mathbf t-\mathbf e_j}$.
		Taking the transpose and applying the positive finite-box equivalence
		therefore identifies $N_{\mathbf t}$ with the module $C_{\mathbf t}$
		defined by
		\begin{equation}\label{eq:no-full-section-source}
			0\longrightarrow S(-\mathbf t)
			\xrightarrow{\ (x_1,\ldots,x_n)^{\mathsf T}\ }
			\bigoplus_{j=1}^n S\bigl(-(\mathbf t-\mathbf e_j)\bigr)
			\longrightarrow C_{\mathbf t}\longrightarrow0.
		\end{equation}
		
		Now put $\mathbf b=(t_1,0,\ldots,0)$ and
		$N_{\mathbf b}=s(T_{\mathbf b})$.
		Both $\mathbf b$ and $\mathbf t$ belong to $U_1(\mathbf t)$, and they
		are distinct because $n\ge2$ and $\mathbf t\ge\mathbf1$.
		The section identity gives
		\[
		H^{n-1}_{\mathfrak m}(N_{\mathbf b})_{\mathbf b-\mathbf1}
		\cong\kk.
		\]
		Consider the \v{C}ech complex of $N_{\mathbf b}$ in degree
		$\mathbf b-\mathbf1=(t_1-1,-1,\ldots,-1)$.
		Since $N_{\mathbf b}$ is nonnegatively graded, every nonzero term must
		invert $x_2,\ldots,x_n$.  Only the terms in cohomological degrees
		$n-1$ and $n$ can therefore be nonzero, and they form the map
		\[
		\bigl((N_{\mathbf b})_{x_2\cdots x_n}\bigr)_{\mathbf b-\mathbf1}
		\longrightarrow
		\bigl((N_{\mathbf b})_{x_1\cdots x_n}\bigr)_{\mathbf b-\mathbf1}.
		\]
		Positive $\mathbf t$-determination identifies this map, up to sign, with
		\[
		x_1:(N_{\mathbf b})_{\mathbf t-\mathbf e_1}
		\longrightarrow (N_{\mathbf b})_{\mathbf t}.
		\]
		Explicitly, the identifications send $u$ on the left to
		$u/(x_2^{t_2+1}\cdots x_n^{t_n+1})$ and $w$ on the right to
		$w/(x_1x_2^{t_2+1}\cdots x_n^{t_n+1})$.
		Consequently,
		\[
		\ker\!\left(
		x_1:(N_{\mathbf b})_{\mathbf t-\mathbf e_1}
		\longrightarrow (N_{\mathbf b})_{\mathbf t}
		\right)\cong\kk.
		\]
		Choose a nonzero vector $v$ in this kernel.  Define a degree-zero map
		\[
		g:\bigoplus_{j=1}^n S\bigl(-(\mathbf t-\mathbf e_j)\bigr)
		\longrightarrow N_{\mathbf b}
		\]
		by sending the first free generator to $v$ and all other free generators
		to zero.  Since $x_1v=0$, this map kills the image of the first map in
		\eqref{eq:no-full-section-source}.  It therefore induces a nonzero map
		$C_{\mathbf t}\to N_{\mathbf b}$.  As $N_{\mathbf t}\cong C_{\mathbf t}$,
		we obtain
		\[
		\Hom_{\mathcal A^+_{\mathbf t}}
		\bigl(s(T_{\mathbf t}),s(T_{\mathbf b})\bigr)\ne0.
		\]
		This contradicts full faithfulness, because the distinct vertex simples
		$T_{\mathbf t}$ and $T_{\mathbf b}$ satisfy
		\[
		\Hom_{\Rep(U_1(\mathbf t))}(T_{\mathbf t},T_{\mathbf b})=0.\qedhere
		\]
	\end{proof}
	
	\begin{remark}\label{rem:no-full-section-intermediate}
		The two arguments use different rigidity properties of sources.
		For $i=1$, certain boundary objects have unique indecomposable sources;
		for $i=n-1$, the transpose determines the indecomposable source of the
		top vertex simple.
		Neither argument extends directly to $2\le i\le n-2$.
		The splitting property used in the first argument already fails when
		$\mathbf t=\mathbf1$.  Indeed, choose $B\subseteq[n]$ with $|B|=i\ge2$
		and put
		\[
		R_B=S/(x_\ell:\ell\in[n]\setminus B),\qquad
		\mathfrak n_B=(x_j:j\in B)R_B.
		\]
		The exact sequence $0\to\mathfrak n_B\to R_B\to\kk\to0$ gives
		\[
		H^i_{\mathfrak m}(\mathfrak n_B)\cong H^i_{\mathfrak m}(R_B).
		\]
		Thus these two objects of $\mathcal A^+_{\mathbf1}$ are sources of the
		same object, but $R_B$ cannot be a direct summand of $\mathfrak n_B$
		because $(\mathfrak n_B)_{\mathbf0}=0$.
		For the second argument, set $A=\Lambda_{\mathbf t}$ and $q=n-i>1$.
		For a finite-dimensional left $A$-module $X$ and $r\ge1$, let
		$\Omega_A^rX$ denote its $r$th syzygy in a minimal projective resolution.
		One then obtains only an embedding
		$\Ext_A^q(X,A)\hookrightarrow\operatorname{Tr}_A(\Omega_A^{q-1}X)$.
		A higher syzygy of an indecomposable module need not be indecomposable,
		and it does not in general determine the original module.
		
		We expect that, for $n\ge2$ and $\mathbf t\ge\mathbf1$,
		$\Phi_i^{\mathbf t}$ admits no fully faithful functorial section for
		any $1\le i<n$, although the arguments above do not settle the
		intermediate degrees $2\le i\le n-2$.
	\end{remark}
	
	\subsection{A fully faithful section in top degree}
	
	The top-degree functor admits a fully faithful section.  Its construction
	uses the reflexivity description in
	Theorem~\ref{thm:reflexive-top-complement}.  As in Section~\ref{sec:finitebox},
	write $L^*=\underline{\Hom}_S(L,S)$ for the graded $S$-dual.
	
	\begin{theorem}\label{thm:top-fully-faithful-section}
		Assume $n\ge1$.  The assignments
		\[
		s_n^{\mathbf t}(N)=(\Dgr N)^*,\qquad
		s_n^{\mathbf t}(f)=(\Dgr f)^*
		\]
		define an additive fully faithful functor
		\[
		s_n^{\mathbf t}:\mathcal H_n(\mathbf t)
		\longrightarrow\mathcal A^+_{\mathbf t}
		\]
		together with a natural isomorphism
		\[
		\Phi_n^{\mathbf t}\circ s_n^{\mathbf t}
		\xRightarrow{\sim}\operatorname{id}_{\mathcal H_n(\mathbf t)}.
		\]
	\end{theorem}
	
	\begin{proof}
		For $N\in\mathcal H_n(\mathbf t)$, put $L=\Dgr N$.
		By Theorem~\ref{thm:reflexive-top-complement}\textup{(ii)}, $L$ is
		reflexive and $L^*\in\mathcal A^+_{\mathbf t}$.
		If $f:N\to N^{\prime}$ is a degree-zero morphism, then
		$\Dgr f:\Dgr N^{\prime}\to\Dgr N$ and
		\[
		(\Dgr f)^*:(\Dgr N)^*\longrightarrow(\Dgr N^{\prime})^*
		\]
		are degree-zero morphisms.  The two contravariant dualities therefore
		define an additive covariant functor with the asserted domain and target.
		
		Graded local duality gives a natural isomorphism
		\[
		\Dgr\Phi_n^{\mathbf t}(s_n^{\mathbf t}(N))
		\cong (s_n^{\mathbf t}(N))^*=L^{**}.
		\]
		The evaluation map $\eta_L:L\to L^{**}$ is a natural isomorphism
		on reflexive modules.  Applying $\Dgr$ and using the natural
		biduality isomorphisms for locally finite graded modules yields
		$\Phi_n^{\mathbf t}(s_n^{\mathbf t}(N))\cong N$, naturally in $N$.
		
		Finally, $S$-duality restricts to a contravariant equivalence on finite
		reflexive graded $S$-modules: their duals are reflexive, and the evaluation
		maps identify their double duals with the original modules.
		For $N,N^{\prime}\in\mathcal H_n(\mathbf t)$, it follows that the map
		on degree-zero Hom spaces induced by $s_n^{\mathbf t}$ is a bijection,
		\[
		\Hom_{\mathcal H_n(\mathbf t)}(N,N^{\prime})
		\xrightarrow{\ \sim\ }
		\Hom_{\mathcal A^+_{\mathbf t}}
		\bigl(s_n^{\mathbf t}(N),s_n^{\mathbf t}(N^{\prime})\bigr).
		\]
		Indeed, graded Matlis duality first identifies the left side with
		$\Hom_S(\Dgr N^{\prime},\Dgr N)$, and $S$-duality then identifies
		this with the right side.  Hence $s_n^{\mathbf t}$ is fully faithful.
	\end{proof}
	
	\section{The global finite-support categories}\label{sec:global-sources}
	
	In this section, we apply the finite-box results to the global categories $\mathcal H_i$.
	We first establish the compatibility with determining boxes, then describe
	the global images and their projective and injective objects.
	
	For $0\le q\le n$, put
	\[
	U_q=\{\mathbf a\in\mathbb N^n: \ell(\mathbf a)\ge q\},
	\qquad
	L_q=\{\mathbf a\in\mathbb N^n: \ell(\mathbf a)=q\}.
	\]
	Thus $U_0=\mathbb N^n$.  We write $\Rep_{\mathrm{fs}}(U_q)$ for the
	category of finite-dimensional representations of $U_q$ with finite
	support.  Recall that $\mathcal H_i=\bigcup_{\mathbf t\in\mathbb N^n}\mathcal H_i(\mathbf t)$
	and that restriction to nonnegative degrees gives an exact equivalence
	\[
	\theta:\mathcal A^-\xrightarrow{\sim}\Rep_{\mathrm{fs}}(\mathbb N^n).
	\]
	Whenever $0\le i<n$, we put $q=n-i$.
	
	\begin{proposition} 
		\label{prop:global-box-compatibility}
		For every $\mathbf t\in\mathbb N^n$ and every $0\le i\le n$, one has
		\[
		\mathcal H_i(\mathbf t)=\mathcal H_i\cap\mathcal A^-_{\mathbf t}
		\]
		as full subcategories of $\mathcal A^-$.
	\end{proposition}
	
	\begin{proof}
		The inclusion from left to right follows from the definitions.
		For the converse, let $N\in\mathcal H_i\cap\mathcal A^-_{\mathbf t}$.
		Choose $\mathbf s\in\mathbb N^n$ with $N\in\mathcal H_i(\mathbf s)$.
		
		Suppose first that $i<n$ and put $q=n-i$.  By
		Theorem~\ref{thm:finite-intrinsic}, $N_{\mathbf a}=0$ for every
		$\mathbf a\in P_{\mathbf s}$ with $\ell(\mathbf a)<q$.
		Moreover, negative $\mathbf s$-determination implies
		$N_{\mathbf a}=0$ for every $\mathbf a\in\mathbb N^n\setminus P_{\mathbf s}$.
		Consequently $N_{\mathbf a}=0$ for every
		$\mathbf a\in\mathbb N^n$ with $\ell(\mathbf a)<q$, in particular for
		such vertices in $P_{\mathbf t}$.  Since $N\in\mathcal A^-_{\mathbf t}$,
		Theorem~\ref{thm:finite-intrinsic}, now applied to $\mathbf t$, gives
		$N\in\mathcal H_i(\mathbf t)$.
		
		If $i=n$, Theorem~\ref{thm:reflexive-top-complement}\textup{(ii)}
		applied to the box $\mathbf s$ shows that $\Dgr N$ is reflexive.
		This is a property of the graded $S$-module $\Dgr N$, independent of
		the determining box.  Applying the same theorem to $\mathbf t$ yields
		$N\in\mathcal H_n(\mathbf t)$.
		All categories involved are full subcategories of $\mathcal A^-$,
		so equality of their objects gives the asserted equality of full subcategories.
	\end{proof}
	
	For $\mathbf t\le\mathbf u$, Proposition~\ref{prop:global-box-compatibility}
	also gives
	\[
	\mathcal H_i(\mathbf u)\cap\mathcal A^-_{\mathbf t}
	=\mathcal H_i(\mathbf t).
	\]
	Thus enlarging the box does not change whether a fixed object of
	$\mathcal A^-_{\mathbf t}$ belongs to the $i$th local-cohomology image.
	In terms of sources, the proposition has the following immediate consequence.
	
	\begin{corollary}\label{cor:source-in-target-box}
		Let $M\in\mathcal A^+$, let $0\le i\le n$, and let
		$\mathbf t\in\mathbb N^n$.  If
		\[
		H^i_{\mathfrak m}(M(-\mathbf1))\in\mathcal A^-_{\mathbf t},
		\]
		then there exists $N\in\mathcal A^+_{\mathbf t}$ such that
		\[
		H^i_{\mathfrak m}(M(-\mathbf1))
		\cong H^i_{\mathfrak m}(N(-\mathbf1)).
		\]
	\end{corollary}
	
	\begin{proof}
		Put $L=H^i_{\mathfrak m}(M(-\mathbf1))$.  Since $M\in\mathcal A^+$,
		we have $L\in\mathcal H_i$.  By the hypothesis and
		Proposition~\ref{prop:global-box-compatibility},
		\[
		L\in\mathcal H_i\cap\mathcal A^-_{\mathbf t}
		=\mathcal H_i(\mathbf t).
		\]
		The definition of $\mathcal H_i(\mathbf t)$ gives the required source $N$.
	\end{proof}
	
	The corollary allows a new source to be chosen in the target determining
	box; it does not imply that the original source $M$ belongs to
	$\mathcal A^+_{\mathbf t}$.  For example, choose
	$\mathbf b\in\mathbb N^n$ with $\mathbf b\not\le\mathbf t$ and put
	$M=S(-\mathbf b)$.  Then $M\notin\mathcal A^+_{\mathbf t}$, although
	\[
	H^i_{\mathfrak m}(M(-\mathbf1))=0\in\mathcal A^-_{\mathbf t}
	\qquad(0\le i<n).
	\]
	However, the complete family of local-cohomology modules
	characterizes the determining bounds of the original module.
	
	\begin{proposition}\label{prop:determination-by-local-cohomology}
		Let $M\in\mathcal A^+$ and $\mathbf t\in\mathbb N^n$.  Then
		\[
		M\in\mathcal A^+_{\mathbf t}
		\quad\Longleftrightarrow\quad
		H^i_{\mathfrak m}(M(-\mathbf1))\in\mathcal A^-_{\mathbf t}
		\text{ for every }0\le i\le n.
		\]
	\end{proposition}
	
	\begin{proof}
		The forward implication is the preservation of determination under
		shifted local cohomology recalled in Section~\ref{sec:determined}.
		For the converse, assume that all the displayed local-cohomology modules
		belong to $\mathcal A^-_{\mathbf t}$.  The assertion is immediate if
		$M=0$, so suppose that $M\ne0$ and take a minimal multigraded free resolution
		\[
		F_\bullet\longrightarrow M,\qquad
		F_r=\bigoplus_{\mathbf b\in\mathbb N^n}
		S(-\mathbf b)^{\beta_{r,\mathbf b}}.
		\]
		This resolution is finite, and only finitely many multidegrees occur.
		Let $\mathbf b$ be a maximal element, in the coordinatewise order,
		of the set of multidegrees occurring in any $F_r$.
		
		Consider the degree $-\mathbf b$ part of the dual complex
		$F_\bullet^*=\underline{\Hom}_S(F_\bullet,S)$.
		A summand $S(\mathbf c)$ contributes to this component only if
		$\mathbf c\ge\mathbf b$, since
		$S(\mathbf c)_{-\mathbf b}=S_{\mathbf c-\mathbf b}$.
		Maximality of $\mathbf b$ forces $\mathbf c=\mathbf b$.
		Moreover, minimality of the resolution makes every differential component
		between equal shifts zero: such a component has degree zero and is
		therefore a scalar, but its entry must lie in $\mathfrak m$.
		Consequently all differentials in $(F_\bullet^*)_{-\mathbf b}$ vanish,
		and for every $r$ one has
		\[
		\underline{\Ext}^{\,r}_S(M,S)_{-\mathbf b}
		\cong\kk^{\beta_{r,\mathbf b}}.
		\]
		Choose $r$ with $\beta_{r,\mathbf b}\ne0$.  Graded local duality gives
		\[
		\Dgr H^{n-r}_{\mathfrak m}(M(-\mathbf1))
		\cong\underline{\Ext}^{\,r}_S(M,S),
		\]
		so $H^{n-r}_{\mathfrak m}(M(-\mathbf1))_{\mathbf b}\ne0$.
		Negative $\mathbf t$-determination of this module implies
		$\mathbf b\le\mathbf t$.
		
		Every degree occurring in the finite resolution is dominated by a
		maximal occurring degree.  Hence every such degree lies in
		$[\mathbf0,\mathbf t]$, and all $F_r$ belong to $\mathcal A^+_{\mathbf t}$.
		Since this category is closed under cokernels,
		$M=\coker(F_1\to F_0)$ also belongs to $\mathcal A^+_{\mathbf t}$.
	\end{proof}
	
	Having established the relation with determining boxes, we now describe
	the global image categories themselves.
	
	\begin{proposition}
		\label{prop:global-image-equivalence}
		For $0\le i<n$, restriction to nonnegative degrees induces an exact
		equivalence
		\[
		\theta|_{\mathcal H_i}:\mathcal H_i
		\xrightarrow{\sim}\Rep_{\mathrm{fs}}(U_{n-i}).
		\]
	\end{proposition}
	
	\begin{proof}
		Since $U_q$ is an upper order ideal, extension by zero identifies
		$\Rep_{\mathrm{fs}}(U_q)$ with the Serre subcategory of
		$\Rep_{\mathrm{fs}}(\mathbb N^n)$ consisting of representations supported
		on $U_q$.  If $L\in\mathcal H_i(\mathbf t)$, then
		Theorem~\ref{thm:finite-intrinsic} shows that $\theta(L)$ is supported on
		$U_q(\mathbf t)$, so $\theta(L)\in\Rep_{\mathrm{fs}}(U_q)$.
		
		Conversely, let $V\in\Rep_{\mathrm{fs}}(U_q)$ and choose a finite box
		$P_{\mathbf t}$ containing its support.  By
		Theorem~\ref{thm:finite-intrinsic}, there is
		$L\in\mathcal H_i(\mathbf t)$ with
		$\theta_{\mathbf t}(L)\cong V|_{U_q(\mathbf t)}$, where both sides are
		regarded as representations supported on $U_q(\mathbf t)$.  Extension by
		zero gives $\theta(L)\cong V$.  Full faithfulness and exactness follow
		from Proposition~\ref{prop:global-fs}.
	\end{proof}
	
	The global categories retain the support description, but their
	projective objects behave differently from those in a finite box.
	
	\begin{theorem}\label{thm:global-no-projectives}
		For $0\le i<n$, one has
		\[
		\operatorname{proj}\mathcal H_i=0.
		\]
	\end{theorem}
	
	\begin{proof}
		By Proposition~\ref{prop:global-image-equivalence}, it suffices to work in
		$\Rep_{\mathrm{fs}}(U_q)$.  Suppose that $V$ is a nonzero projective
		object.  Choose $\mathbf t$ strictly larger in every coordinate than
		every vertex in its support.  Then $V_{\mathbf t}=0$.
		
		Extension by zero from $U_q(\mathbf t)$ to $U_q$ is exact and fully
		faithful.  Hence the lifting property of $V$ implies that
		$V|_{U_q(\mathbf t)}$ is projective in $\Rep(U_q(\mathbf t))$.
		Every nonzero projective in this finite category is a direct sum of
		principal upper-interval representations and therefore has a nonzero
		component at $\mathbf t$.  This contradicts $V_{\mathbf t}=0$.
	\end{proof}
	
	For $\mathbf a\in U_q$, let $P_{\mathbf a}$ be the representation with
	value $\kk$ at each vertex $\mathbf b\ge\mathbf a$, zero elsewhere, and
	identity maps between nonzero components.  Although $P_{\mathbf a}$ has
	infinite support, its restriction $P_{\mathbf a}^{\mathbf t}$ to
	$U_q(\mathbf t)$ is projective whenever $\mathbf a\le\mathbf t$.
	Principal lower intervals, on the other hand, have finite support and
	provide the injective objects.  For a poset $P$ and $a\in P$, we write
	$\kk\Hom_P(-,a)$ for the contravariant functor whose value at $b\in P$
	is the free $\kk$-vector space with basis $\Hom_P(b,a)$, with maps
	induced by composition.  For $\mathbf a\in\mathbb N^n$, write
	\[
	I_{\mathbf a}=D\circ\bigl(\kk\Hom_{\mathbb N^n}(-,\mathbf a)\bigr).
	\]
	This is the principal lower-interval representation: its component at
	$\mathbf b$ is $\kk$ if $\mathbf b\le\mathbf a$ and is zero otherwise,
	with identity maps between nonzero components.  For $\mathbf a\in U_q$,
	put
	\[
	I_{\mathbf a}^{(q)}=I_{\mathbf a}|_{U_q}
	=D\circ\bigl(\kk\Hom_{U_q}(-,\mathbf a)\bigr).
	\]
	In particular, $I_{\mathbf a}^{(0)}=I_{\mathbf a}$.
	
	\begin{proposition}\label{prop:global-injectives}
		For $0\le i<n$, the category $\mathcal H_i$ has enough injectives.
		Under the equivalence of Proposition~\ref{prop:global-image-equivalence},
		its indecomposable injectives correspond to
		$I_{\mathbf a}^{(q)}$, where $q=n-i$ and $\mathbf a\in U_q$.
	\end{proposition}
	
	\begin{proof}
		We prove the representation-theoretic assertion for $0\le q\le n$.
		For $\mathbf a\in U_q$, the support of $I_{\mathbf a}^{(q)}$ is the
		finite set $\{\mathbf b\in U_q:\mathbf b\le\mathbf a\}$.  The natural
		isomorphism
		\[
		\Hom_{\Rep_{\mathrm{fs}}(U_q)}(V,I_{\mathbf a}^{(q)})
		\cong D(V_{\mathbf a})
		\]
		shows that $I_{\mathbf a}^{(q)}$ is injective.  Its endomorphism ring is
		$\kk$, so it is indecomposable.
		
		Let $V\in\Rep_{\mathrm{fs}}(U_q)$.  For each vertex $\mathbf a$ in its
		support, choose a basis of $D(V_{\mathbf a})$.  The corresponding maps
		to $I_{\mathbf a}^{(q)}$ give a morphism
		\[
		V\longrightarrow
		\bigoplus_{\mathbf a\in\supp V}
		\bigl(I_{\mathbf a}^{(q)}\bigr)^{\dim_\kk V_{\mathbf a}}.
		\]
		At each vertex, the maps indexed by that vertex separate vectors, so
		this morphism is injective.  Thus the category has enough injectives.
		Every injective is a direct summand of such a finite sum, and the
		Krull--Schmidt theorem shows that the displayed objects exhaust the
		indecomposable injectives.  Transporting these assertions through
		$\theta$ proves the result.
	\end{proof}
	
	\begin{remark}
		The representation-theoretic assertions above hold more generally.
		Let $P$ be an inward finite poset, meaning that every principal lower
		set $P_{\le a}$ is finite. Then $\Rep_{\mathrm{fs}}(P)$ has enough
		injectives, and its indecomposable injectives are precisely
		$D\circ\bigl(\kk\Hom_P(-,a)\bigr)$ for $a\in P$. Indeed, the proof of
		Proposition~\ref{prop:global-injectives} applies verbatim, using
		$\Hom(V,D\circ\bigl(\kk\Hom_P(-,a)\bigr))\cong D(V_a)$.
		If, in addition, every principal upper set $P_{\ge a}$ is infinite,
		then $\Rep_{\mathrm{fs}}(P)$ has no nonzero projective objects.
		Inward finiteness alone does not imply this latter assertion, as finite
		posets show. The posets $U_q$ satisfy both conditions, so the preceding
		results follow from these general properties together with
		Proposition~\ref{prop:global-image-equivalence}.
	\end{remark}
	
	We now turn to the top degree.  Let $\mathcal C^{\mathrm{fs}}_n$ be the
	full subcategory of
	$\operatorname{Rep}_{\mathrm{fs}}(\mathbb N^n)$ consisting of those $V$ for
	which there is an exact sequence
	\[
	0\longrightarrow W\longrightarrow I^0\longrightarrow I^1
	\longrightarrow V\longrightarrow0,
	\]
	where $I^0$ and $I^1$ are finite direct sums of the $I_{\mathbf a}$.
	By the case $q=0$ of the preceding proof, these are precisely the
	injectives in $\Rep_{\mathrm{fs}}(\mathbb N^n)$.
	
	\begin{theorem}\label{thm:global-top-cosyz}
		Under the equivalence $\theta$, one has
		\[
		\theta(\mathcal H_n)=\mathcal C^{\mathrm{fs}}_n.
		\]
		Thus $\mathcal H_n$ is equivalent to the category of second cosyzygies
		in $\Rep_{\mathrm{fs}}(\mathbb N^n)$.
	\end{theorem}
	
	\begin{proof}
		Let $L\in\mathcal H_n$ and choose $\mathbf t$ such that
		$L\in\mathcal H_n(\mathbf t)$.  By Theorem~\ref{thm:top-finite},
		$\theta_{\mathbf t}(L)$ is the cokernel of a map between finite direct
		sums of principal injectives on $P_{\mathbf t}$.  Extending by zero
		sends each such injective at $\mathbf a$ to $I_{\mathbf a}$ and
		preserves exactness.  Hence $\theta(L)\in\mathcal C^{\mathrm{fs}}_n$.
		
		Conversely, let $V\in\mathcal C^{\mathrm{fs}}_n$.  The injectives in a
		defining sequence have support in a common finite box $P_{\mathbf t}$,
		as do their subobjects and quotients.  Restricting the sequence to this
		box and applying Theorem~\ref{thm:top-finite} gives
		$V|_{P_{\mathbf t}}\in\mathcal T_n(\mathbf t)$.  Thus
		$\theta^{-1}(V)\in\mathcal H_n(\mathbf t)\subseteq\mathcal H_n$.
	\end{proof}
	
	\begin{corollary}\label{cor:global-top}
		The category $\mathcal H_n$ is extension-closed in $\mathcal A^-$, and
		\[
		\mathcal H_i\cap\mathcal H_n=0
		\qquad(0\le i<n).
		\]
		If $n\ge1$, then $\mathcal H_n$ is not an abelian category.
	\end{corollary}
	
	\begin{proof}
		For an extension with end terms in $\mathcal H_n$, choose a common
		determining box and apply Theorem~\ref{thm:top-finite}.  The intersection
		statement follows similarly from Proposition~\ref{prop:top-disjoint}.
		
		Assume $n\ge1$.  In the representation category, consider the canonical
		epimorphism $I_{\mathbf e_1}\to I_{\mathbf0}=S_{\mathbf0}$.
		Both objects belong to $\theta(\mathcal H_n)$.  Let $K$ be its kernel.
		Then $K\ne0$ and $K_{\mathbf0}=0$.  If
		$V\in\theta(\mathcal H_n)$, choose $\mathbf t$ and
		$L\in\mathcal H_n(\mathbf t)$ such that $V\cong\theta(L)$; this is
		possible because $\mathcal H_n=\bigcup_{\mathbf t\in\mathbb N^n}\mathcal H_n(\mathbf t)$.
		Theorem~\ref{thm:top-finite} shows that every structure map
		$V_{\mathbf0}\to V_{\mathbf b}$ is surjective.  Hence any morphism
		$V\to K$ is zero: its component at $\mathbf0$ is zero, and naturality
		together with surjectivity gives zero at every other vertex.  Thus the
		displayed epimorphism is monic in
		$\theta(\mathcal H_n)$.  It is epic in the ambient representation category,
		but is not an isomorphism.  Since abelian categories are balanced,
		$\mathcal H_n$ is not abelian.
	\end{proof}
	
	The object $\mathcal E_{\mathbf0}=\sigma_{\mathbf t}S_{\mathbf0}$
	is independent of $\mathbf t$ and will also be viewed as an object of
	$\mathcal A^-$ in the next statement.
	We use the notation $\bullet$ introduced before
	Proposition~\ref{prop:top-penultimate}.
	
	\begin{corollary}\label{cor:global-top-extension}
		Assume $n\ge1$.  Then
		\[
		\mathcal A^- = \mathcal H_{n-1}\bullet\operatorname{add}\mathcal E_{\mathbf0}
		= \mathcal H_{n-1}\bullet\mathcal H_n,
		\]
		whereas $(\mathcal H_{n-1},\mathcal H_n)$ is not a torsion pair.
	\end{corollary}
	
	\begin{proof}
		The decompositions follow from Proposition~\ref{prop:top-penultimate}
		by choosing a common finite box.  For any $\mathbf a\ne\mathbf0$, the
		vertex simple $S_{\mathbf a}$ lies in $\theta(\mathcal H_{n-1})$, while
		$I_{\mathbf a}$ lies in $\theta(\mathcal H_n)$.  The nonzero inclusion
		$S_{\mathbf a}\hookrightarrow I_{\mathbf a}$ shows that
		$\Hom(\mathcal H_{n-1},\mathcal H_n)\ne0$, so the pair is not a torsion
		pair.
	\end{proof}
	
	Finally, the support-rank filtration of
	Corollary~\ref{cor:canonical-layer-filtration} also passes to the global category.
	Let $L\in\mathcal H_r$, where $0\le r<n$, and put $V=\theta(L)$.
	For $0\le s\le r$, define
	\[
	(T_sV)_{\mathbf a}=
	\begin{cases}
		V_{\mathbf a},&\ell(\mathbf a)\ge n-s,\\
		0,&\ell(\mathbf a)<n-s.
	\end{cases}
	\]
	Since $U_{n-s}$ is an upper order ideal, these components define a
	subrepresentation of $V$.  They give a functorial filtration
	\[
	0=T_{-1}V\subseteq T_0V\subseteq\cdots\subseteq T_rV=V,
	\]
	whose factor $T_sV/T_{s-1}V$ is supported on $L_{n-s}$.  All terms and
	factors have finite support.  Transporting the filtration through
	$\theta^{-1}$ gives the corresponding filtration in $\mathcal H_r$.
	
	\section{Open problems and extensions}\label{sec:open-problems}
	
	The results above concern local cohomology supported at $\mathfrak m$.
	We conclude with two problems asking how the image categories and their
	sources change when the support ideal is replaced by another monomial
	ideal.  These questions arise naturally from our results and are also
	motivated by the work of
	Eisenbud--Musta\c{t}a--Stillman~\cite{EMS} on multigraded local cohomology
	with monomial supports.
	
	\subsection{Complete-intersection monomial supports}
	
	Let
	\[
	I=(u_1,\ldots,u_c)\subseteq S
	\]
	be generated by a monomial regular sequence.  Fix $\mathbf t\in\mathbb N^n$
	and $i\ge0$, and consider
	\[
	\Phi_{I,i}^{\mathbf t}:\mathcal A^+_{\mathbf t}\longrightarrow
	\mathsf{Gr}_{\mathbb Z^n}(S),
	\qquad
	M\longmapsto H_I^i(M(-\mathbf1)).
	\]
	
	\begin{problem}
		For every such $I$, $\mathbf t$, and $i$, investigate the following questions.
		\begin{enumerate}
			\item Is the essential image of $\Phi_{I,i}^{\mathbf t}$ an abelian
			category?  Determine precisely when this holds and when the image is a
			Serre subcategory of a suitable ambient abelian category.
			\item Can the essential image be described by an $I$-dependent finite
			multidegree window $W_{I,\mathbf t}$ and a finite combinatorial category
			$\mathcal C_{I,\mathbf t,i}$?
			\item When the image is abelian, classify its indecomposable projective
			and injective objects.
			\item Can source modules be chosen functorially?  When can such a choice
			be fully faithful or, if the image is abelian, exact?  Determine formulas
			for annihilators and supports, and for attached primes when the image
			objects are Artinian.
		\end{enumerate}
		
		A first example of a non-maximal monomial complete-intersection ideal we may consider is
		\[
		I=(x_1x_2,x_3x_4)\subseteq\kk[x_1,x_2,x_3,x_4].
		\]
	\end{problem}
	
	\subsection{General monomial supports}
	
	For general monomial supports, one can ask more broadly which modules
	occur as local cohomology modules of finitely generated multigraded
	modules, without fixing a finite box.
	
	\begin{problem}
		Let $I\subseteq S$ be a monomial ideal and let $i\ge0$.
		How can one characterize the $\mathbb Z^n$-graded $S$-modules $N$ for which
		there exists a finitely generated $\mathbb Z^n$-graded $S$-module $M$
		such that $N\cong H_I^i(M)$?  How does the combinatorial structure of 
		$\sqrt I$ enter this characterization and the construction of such
		source modules $M$?
	\end{problem}
	
	\section{Conclusion}
	
	We have described the non-top finite-box local-cohomology categories
	by upper support-rank posets, and the top-degree category by second
	cosyzygies.  We have also determined attached primes
	in the image categories and computed Bass numbers of their projective and
	injective objects.  In every non-top degree the local-cohomology
	functor has an exact functorial section of minimal projective dimension, while
	for $n\ge2$ and $\mathbf t\ge\mathbf1$ no fully faithful section exists
	for $i=1$ or $i=n-1$.  In top degree, we construct a fully faithful
	functorial section by composing graded Matlis duality with $S$-duality.
	The global non-top categories have no nonzero projective objects but have
	enough injectives, while the top category is described by finite-support
	second cosyzygies.
	
	The problems in Section~\ref{sec:open-problems} concern the structure of
	finite-box image categories for complete-intersection monomial supports
	and, more generally, the realization of multigraded modules as local
	cohomology modules with monomial support and the construction of their
	sources.
	
	\section*{Acknowledgements}
	
	The author used ChatGPT (OpenAI; GPT-5.6 Sol and GPT-6 Astra) in  September--October 2026
	accessed via \mbox{\url{https://chatgpt.com}} to assist
	with algebraic reasoning, and with the drafting and revision of parts of the manuscript. However, all research questions were formulated by the author, who independently verified the
	mathematical arguments and takes full responsibility for the mathematical
	results, proofs, references, and final text.
	
	\medskip\noindent\textbf{Competing interests.}
	The author declares none.

\end{document}